\documentclass[12pt]{amsart}
\usepackage{amsmath, amsthm, mathrsfs, amssymb, verbatim, bbm, url, appendix}
\usepackage{amscd}
\usepackage{latexsym,amssymb}
\usepackage{graphicx}
\usepackage{color}
\usepackage{enumerate}
\allowdisplaybreaks[3]

\usepackage{latexsym}
\usepackage{epsfig}
\usepackage{amsfonts}
\usepackage{enumerate}
\usepackage{times}

\newtheorem{maingt}{Theorem}[section]
\newtheorem{tspace1}[maingt]{Definition}

\newtheorem{dilatoninput}{Remark}[section]
\newtheorem{cohtwfake}[dilatoninput]{Example}
  \newtheorem{cobcg}[dilatoninput]{Theorem}
  \newtheorem{mclaurin}[dilatoninput]{Definition}

\newtheorem{euler1}{Remark}[section]
\newtheorem{euler}[euler1]{Assumption}
\newtheorem{inv}[euler1]{Remark}
\newtheorem{hypers}[euler1]{Remark}
\newtheorem{positive}[euler1]{Convention}
\newtheorem{main}[euler1]{Theorem}
 \newtheorem{rem}[euler1]{Remark}
 \newtheorem{rem11}[euler1]{Remark}
 
 \newtheorem{fakest}{Proposition}[section]
 \newtheorem{identification}[fakest]{Remark}
 \newtheorem{stemtw}[fakest]{Proposition}
 \newtheorem{cohstem}[fakest]{Proposition}
 \newtheorem{stems}[fakest]{Proposition}
 \newtheorem{psic}[fakest]{Proposition}
 \newtheorem{11}[fakest]{Proposition}
 
 \newtheorem{snj}{Theorem}[section]
  
    \newtheorem{twsnj}[snj]{Theorem}
    \newtheorem{globaltwsnj}[snj]{Theorem}
     \newtheorem{shift}[snj]{Remark}
     \newtheorem{general}[snj]{Corollary}
     \newtheorem{dmodule}[snj]{Theorem}
     \newtheorem{lefschetz}[snj]{Theorem}
     \newtheorem{projhyper}[snj]{Corollary}
     \newtheorem{cigeneral}[snj]{Remark}
     \newtheorem{toricfibration}[snj]{Theorem}
     \newtheorem{last}[snj]{Proposition}

\theoremstyle{remark}

\theoremstyle{remark}

\numberwithin{equation}{section}
\numberwithin{equation}{section}

\newcommand{\ra}{\right\rangle}
\newcommand{\la}{\left\langle}

\newcommand{\iy}{\mathcal{Y}}

\def\<{\left\langle}
\def\>{\right\rangle}
\begin{document}

\title{Twisted K-theoretic Gromov-Witten invariants}
\author[Tonita]{Valentin Tonita}
\address{Humboldt Univerit\"at \\ Rudower Chaussee 25 Raum 1425 \\ Berlin 10099\\ Germany}
\email{valentin.tonita@hu-berlin.de}

% \date{\today}
\begin{abstract}
We introduce twisted K-theoretic Gromov-Witten (GW) invariants in the frameworks of both "ordinary"  and permutation-equivariant K-theoretic GW theory defined recently by Givental. We focus on the case when the twisting is given by the Euler class of an index bundle which allows one (under a convexity assumption on the bundle) to relate K-theoretic GW invariants of hypersurfaces with those of the ambient space. Using the methods developed in \cite{gito} we characterize the range of the $J$-function of the twisted theory in terms of the untwisted theory. 
As applications we use the $\mathcal{D}_q$ module structure in the permutation-equivariant case to generalize results of \cite{giv}: we prove a general "quantum Lefschetz" type theorem for complete intersections given by zero sections of convex vector bundles  and we relate points on the cones of the total space with those of the base of a toric fibration.  
\end{abstract}
\maketitle

\section{Introduction}
 K-theoretic Gromov-Witten invariants have been introduced by A. Givental and Y.-P. Lee(\cite{giv_lee}, \cite{ypl}). They are holomorphic Euler characteristics of vector bundles on the moduli spaces of stable maps to $X$.  

 One of the initial motivations was to provide a new way to count rational curves on the famous Calabi-Yau quintic threefold. In this case  the moduli spaces of stable maps are zero dimensional and the Euler characteristics of their structure sheaves would  give the number of curves. For this purpose a formalism of twisted K-theoretic Gromov-Witten theory analogous to the cohomological one needs to be developed. Roughly speaking the twisted invariants are defined by including in the  correlators Euler classes of index bundles. Then one would like to express them in terms of the original ("untwisted") invariants. If one starts with a convex line bundle the genus zero twisted invariants are actually GW invariants of the hypersurface given as the zero locus of a section of the bundle. Thus the relation between the twisted and untwisted theories translates into a relation between GW invariants of the hypersurface and those of the ambient space. 
 
 Givental recently introduced in \cite{giv} a new enriched version of the theory called permutation-equivariant K-theoretic Gromov-Witten theory, which takes into account the $S_n$ action on the moduli spaces permuting the marked points. The permutation-equivariant theory fits better within the framework of mirror symmetry. In particular certain $q$-hypergeometric series associated with toric manifolds lie on the cone of their permuation-equivariant K-theory.
 
 In this paper we introduce twisted K-theoretic Gromov-Witten invariants in genus zero for both versions of quantum K-theory. The generating functions 
 of the invariants are called the $J$-functions. Their images are Lagrangian cones  living in infinite dimensional loop spaces $\mathcal{K}$. We use Kawasaki Riemann Roch theorem to characterize the cone of the twisted theory in terms of the cone of the untwisted theory. 
 
  Our results allow us to significantly generalize Givental's K-theoretic "mirror formulas" for the permutation-equivariant theories: we prove a "quantum Lefschetz" type theorem for an arbitrary smooth projective variety $X$, namely we show that certain hypergeometric modifications of points on the cone of $X$ lie on the cone of a complete intersection given as the zero locus of sections of convex line bundles. Moreover we prove a similar statement for toric fibrations, describing points on the cone of the total space $E$ in terms of points on the base.
 
 We include some short computations at the end. In particular we show how one can count the $2875$ lines on the Calabi-Yau quintic using our results. More generally one can use our results to compute GW invariants of all complete intersections in projective spaces. According to reconstruction theorems (see \cite{giv}) one can recover all genus zero "ordinary" and permutation-equivariant K-theoretic GW invariants of a projective manifold (under the assumption that the ring $K^0(X)$ is generated by line bundles) from a point on their K-theoretic Lagrangian cone. Hence in principle we can determine the quantum K-theory for all projective manifolds for which our results can be used to find a point on their K-theoretic Lagrangian cones. 
 
 The paper is organized as follows. In Section \ref{sec02} we introduce the basic objects of K-theoretic \footnote{when not specified it will always mean non permutation-equivariant} GW theory in genus zero including the $J$-function and the Lagrangian cone $\mathcal{L}$. We also recall the main result of \cite{gito} which characterizes the cone $\mathcal{L}$ in terms of the cone of the cohomological GW theory.  Section \ref{twcohgwgen} contains a brief review of the formalism of cohomological twisted GW theory. In Section \ref{sec04}  we define twisted K-theoretic invariants and state the main result in the non-permutation equivariant case. This is Theorem \ref{thm12} which describes the cone $\mathcal{L}^{tw}$ in terms of the cone of the untwisted one. The proof of Theorem \ref{thm12} is done in Section \ref{sec05} following the ideas in \cite{gito}. We write Laurent expansions of the $J$ function near each root of unity and identify it with generating series coming from certain twisted cohomological theories to characterize their ranges. 
 Section \ref{sec06} deals with the permutation-equivariant theory: the relation between the twisted and untwisted cones  turns out "nicer" in this case. This is the content of Theorem \ref{snthm}. We combine it with the $\mathcal{D}_q$ module structure of \cite{giv} to prove the applications mentioned above. We end with some computations of  K-theoretic GW invariants of complete intersections in projective spaces.

 \textbf{Acknowledgements.} I would like to thank A. Givental for explaining to me the permutation-equivariant theory and for useful discussions. Parts of these discussions took place during a visit at IBS, Center for Geometry and Physics, Pohang. I would like to thank the institute for support and hospitality. I would also like to thank the anonymous referee for the useful comments.
 
 \section{K-theoretic Gromov-Witten theory} \label{sec02}
 Let $X$ be a complex projective manifold. We denote by $X_{0,n,d}$ Kontsevich's moduli stack of genus zero stable maps to $X$: they parametrize
  data $(C, x_1,\ldots ,x_n, f)$  such that
  \begin{itemize}
  \item $C$ is a connected projective complex curve of arithmetic genus zero with at most nodal singularities.
  \item  $(x_1,\ldots x_n) \in C$ is an ordered $n$-tuple of distinct smooth points on $C$ (they are called marked points).
  \item $f:C\to X$ is a map of degree $d\in H_2(X,\mathbb{Z}$).
  \item The data $(C, x_1,\ldots ,x_n, f)$ has finite automorphism group, where an automorphism is defined to be an automorphism $\varphi:C\to C$ such 
  that $\varphi(x_i) =x_i$ for all $i=1,..n$ and $f\circ \varphi =f$. 
  \end{itemize}
   For each $i=1,..,n$ there are evaluation maps $ev_i : X_{0,n,d}\to X $ defined by sending a point $(C, x_1,\ldots ,x_n, f)\mapsto f(x_i)$ and   cotangent line bundles $L_i \to X_{0,n,d}$ whose fibers over a point $(C, x_1,\ldots ,x_n, f)$ are identified with $T^\vee_{x_i}C$.
  
  K-theoretic Gromov-Witten (GW) invariants have been defined by Givental and Lee (\cite{giv_lee},\cite{ypl}) as sheaf holomorphic Euler characteristics on $X_{0,n,d}$  obtained using the maps $ev_i$ and the line bundles $L_i$:
   \begin{align*}
   \chi\left(X_{0,n,d}, \mathcal{O}_{n,d,X}^{vir}\otimes _{i=1}^n ev_i^* (E_i)L_i^{k_i}\right)\in \mathbb{Z}.
   \end{align*}
 
 Here $\mathcal{O}_{n,d}^{vir}\in K_0(X_{0,n,d})$ (we will generally suppress $X$ from the notation) is the virtual structure sheaf defined in \cite{ypl}.
  We will use correlator notation for the invariants:
   \begin{align*}
   \la E_1 L^{k_1},\ldots , E_n L^{k_n}\ra_{0,n,d}:= \chi\left(X_{0,n,d}, \mathcal{O}_{n,d}^{vir}\otimes _{i=1}^n ev_i^* (E_i)L_i^{k_i}\right).
   \end{align*}
   
   The generating series of these invariants is called the (K-theoretic) $J$-function. Let 
   \begin{align*}
   &\mathcal{K}_+ : = K^0(X,\mathbb{C}[[Q]])\otimes \mathbb{C}[q,q^{-1}], \\
   &\mathcal{K}:=K^0(X,\mathbb{C}[[Q]])\otimes \mathbb{C}(q).
   \end{align*}  
   The $J$-function is 
    \begin{align*}
   & \mathcal{J}:\mathcal{K}_+ \to \mathcal{K}, \\
   & \mathcal{J}(\mathbf{t}(q))=1-q+\mathbf{t}(q) + \sum_{d,n,a}\frac{Q^d}{n!}\Phi^a\la \frac{\Phi_a}{1-qL}, \mathbf{t}(L),\ldots , \mathbf{t}(L) \ra_{0,n+1,d} . 
    \end{align*}
    Here $\{\Phi_a\}, \{\Phi^a\}$ are bases of $K^0(X)$ dual with respect to the pairing
      \begin{align*}
      (\Phi_a, \Phi_b)= \chi(X, \Phi_a\otimes \Phi_b)
      \end{align*}
     and $Q^d$ are monomials in the Novikov ring based on the cone of effective curves in $Eff(X)\subset H_2(X)$.
   
   The image $\mathcal{L}\subset \mathcal{K}$ of the $J$ function has been characterized in \cite{gito} in terms of the cohomological GW theory of $X$. We briefly recall the main result there, referring to \cite{gito} for details. 
     
    To express holomorphic Euler characteristics of a vector bundle $V$ on a compact complex orbifold $\iy$ as a cohomological integral one uses Kawasaki Riemann Roch (KRR) theorem of \cite{kawasaki} (proven by T\"oen in \cite{toen}  for proper smooth Deligne-Mumford stacks). The integrals are supported on the inertia orbifold $I\iy$ of $\iy$:
       \begin{align}
        \chi(\iy, V) = \sum_\mu \int_{\iy_\mu} \operatorname{Td}(T_{\iy_\mu})\operatorname{ch} \left(\frac{\operatorname{Tr}(V)}{\operatorname{Tr}(\Lambda^\bullet N^\vee_\mu)}\right).\label{eqn:KRR} 
        \end{align}
    We now explain this ingredients of this formula. $I\iy$ is the inertia orbifold of $\iy$, given set-theoretically by pairs $(y, (g))$, where $y\in \iy$ and $(g)$ is (the conjugacy class of) a symmetry which fixes $y$.  We denote by $\iy_\mu$ the connected components \footnote{We frequently refer to them as Kawasaki strata.} of $I\iy$.

        For a vector bundle $V$, let $V^\vee$ be the dual bundle to $V$. The restriction of $V$ to $\iy_\mu$ decomposes in characters of the $g$ action. Let $V_r^{(l)}$ be the subbundle of the restriction of $V$ to $\iy_\mu$ on which $g$ acts with eigenvalue $e^{\frac{2\pi i l}{r}}$. Then the  trace $\operatorname{Tr}(V)$ is defined to be the orbibundle whose fiber over the point $(p, (g))$ of $\iy_\mu$ is 
         \begin{align*}
      \operatorname{Tr}(V):= \sum_{0\leq l\leq r-1} e^{\frac{2\pi i l}{r}} V^{(l)}_r . 
         \end{align*} 
      Finally, $\Lambda^\bullet N^\vee_\mu$ is the K-theoretic Euler class of the normal bundle $N_\mu$ of $\iy_\mu$ in $\iy$. $\operatorname{Tr}(\Lambda^\bullet N^\vee_\mu)$ is invertible because the symmetry $g$ acts with eigenvalues different from $1$ on the normal bundle to the fixed point locus.

    $X_{0,n,d}$ is not smooth but it has a perfect obstruction theory which can be used to define its virtual fundamental class (see \cite{behfan}). For a stack $(\mathcal{Y},E^\bullet)$ with a perfect obstruction theory that can be embedded in a smooth proper stack ($X_{0,n,d}$ satisfies this assumption) one can choose an explicit resolution of $E^\bullet$ as a complex of vector bundles $E^{-1}\to E^0$. Let $E_0\to E_1$ be the dual complex. Then the virtual tangent bundle of $\mathcal{Y}$ can be defined as the class $[E_0] \ominus [E_1] \in K^0(\mathcal{Y})$ (see \cite{fago}).
    Moreover the connected components of the inertia orbifold of $\mathcal{Y}$ inherit perfect obstruction theories which can be used to define their virtual normal bundles.
     
    It was proved in \cite{to_k} that one can apply KRR theorem to the moduli spaces $X_{0,n,d}$ by replacing all the ingredients in the formula with their virtual counterparts. The symmetries on $X_{0,n,d}$ which have non-trivial action on the cotangent line bundle $L_1$ create poles at all roots of unity in the 
    $J$-function. For each primitive root of unity $\eta$ of order $m$ denote by $\mathcal{J}_\eta$ the Laurent expansion of the $J$-function in $(1-q\eta)$ and regard it as an element in the loop space of  such Laurent power series  with coefficients in $K^0(X)$
    \begin{align*}
    \mathcal{K}^\eta:= K^0(X)[\frac{1}{1-q\eta},(1-q\eta)]].
    \end{align*} 
    Let us look at $\eta=1$:  the contributions in KRR formula come from the identity component of the inertia orbifold of $X_{0,n,d}$. They were called {\em fake} K-theoretic GW invariants and are of the form: 
      \begin{align*}
       \la \mathbf{t}(L),\ldots ,\mathbf{t}(L)\ra_{0,n,d}^{fake} :=\int_{[X_{0,n,d}]} \prod_{i=1}^n\operatorname{ch}(\mathbf{t}(L_i))\operatorname{Td}(T_{0,n,d}). 
          \end{align*} 
          where $[X_{0,n,d}]$ is the virtual fundamental class of the moduli space, $T_{0,n,d}$ is the virtual tangent bundle and the product is the cohomological cup product.
          
      Consider the generating series of the fake invariants, i.e. elements of $\mathcal{K}^1$ of the form
      \begin{align*}
    J_{fake}(q,\mathbf{t}(q)):=  1-q+\mathbf{t}(q)+ \sum_{d,n,a}\frac{Q^d}{n!}\Phi^a\la \frac{\Phi_a}{1-qL}, \mathbf{t}(L),\ldots , \mathbf{t}(L) \ra^{fake}_{0,n+1,d} .
      \end{align*}
      Here the argument $\mathbf{t}(q)$ belongs to the space 
      \begin{align*}
      \mathcal{K}^1_+ :=K^0(X, \mathbb{C}[[Q]])[[q-1]].
      \end{align*}
      
   The range of $J_{fake}$ spans a Lagrangian cone $\mathcal{L}_{fake}\subset \mathcal{K}^1$ which can be described explicitly in terms of the cohomological GW theory of $X$. We will make this precise in the next section.
  
    The main theorem of \cite{gito} describes for all $\eta$ the range of $\mathcal{J}_\eta(\mathbf{t}(q))$ in terms of the cone $\mathcal{L}_{fake}$.
    
    \begin{maingt}\cite{gito}\label{mainadelic}
    The  K-theoretic $J$ function of $X$ is completely characterized by the following conditions
    \begin{enumerate}
    \item If $\eta$ is not a root of unity $\mathcal{J}_\eta (\mathbf{t}(q))$ does not have poles at $q=\eta^{-1}$.
    \item $\mathcal{J}_1(\mathbf{t}(q)) \in \mathcal{L}_{fake}$.
    \item  Let $\eta$ be a  primitive root of unity of order $k\neq 1$ and let $\mathcal{T}_k (\mathcal{J}_1(0))$ be as in Definition \ref{tang1} below. Identify 
    $\mathcal{K}^\eta$ with $\mathcal{K}^1$ via $q\eta\mapsto q$. Then 
    \begin{align*}
    \mathcal{J}_\eta (\mathbf{t}(q^{1/k}\eta^{-1}))\in \exp \sum_{i \geq 1} \left(\frac{\psi^i T_X^\vee}{i(1-\eta^{-i} q^{i/k})}-\frac{\psi^{ik} T_X^\vee}{i(1-q^{ik})} \right)\mathcal{T}_k (\mathcal{J}_1(0)).
    \end{align*}
   
      \end{enumerate}
    \end{maingt}
    \begin{tspace1}\label{tang1}
             {\em  Let $\mathbf{f}$ be a point on  $\mathcal{L}_{fake}$, let $T(\mathbf{f})$ be the tangent space to $\mathcal{L}_{fake}$ at $\mathbf{f}$, considered as
                    the image of a map $S(q,Q):\mathcal{K}^1_+\to \mathcal{K}^1$. Recall the Adams operations $\psi^k$ are ring isomorphisms of $K^0(X)$ which act on line bundles as $L\mapsto L^k$. Denote by $\psi^{\frac{1}{k}}$  the isomorphism of $\mathcal{K}^1_+$ which is the inverse of $\psi^k$ on $K^0(X)$ and does not act on $q, Q^d $. Then define  }
                  \begin{align*}
                  \mathcal{T}_k (\mathbf{f}) :=\quad  \operatorname{Image} \quad \operatorname{of} \quad \psi^k \circ  S(q^k, Q^k) \circ \psi^{1/k}: \mathcal{K}^1_+ \to \mathcal{K}^1.
                  \end{align*}              
                 \end{tspace1}
                 
                 \section{Twisted cohomological Gromov-Witten theory}\label{twcohgwgen} 
                         
                         The proofs of Theorem \ref{mainadelic} as well as of the main statements in the upcoming sections rely heavily on the machinery of 
                         twisted cohomological GW invariants. They were introduced in \cite{coatesgiv} and generalized in various directions in \cite{tseng} and \cite{to1}.
                         We succinctly review it below, emphasizing the example of the fake GW invariants.  We use the same correlator notation 
                         for cohomological GW invariants 
                          $$\langle \varphi_1 \psi^{k_1},\ldots ,\varphi_n \psi^{k_n}\rangle_{0,n,d}:= \int_{[X_{0,n,d}]}\prod_{i=1}^n ev_i^*(\varphi_i)\psi_i^{k_i}.$$ 
                         The different notation for the classes inside the correlators makes it easy to distinguish them from the K-theoretic GW-invariants. The product in the integrand is the cohomological cup product.
                          
                           Let $\mathcal{H}$ be the loop space of the cohomological GW theory of $X$
                              \begin{align*}
                              \mathcal{H}:=H^*(X,\mathbb{C}[[Q]])[z^{-1},z]].
                               \end{align*}
                              It comes equipped with a symplectic form and carries a distinguished polarization  $ \mathcal{H}_+ \oplus \mathcal{H}_-$  where
                              \begin{align*}
                             \mathcal{H}_+:= H^*(X,\mathbb{C}[[Q]])[[z]], \quad \mathcal{H}_-:= \frac{1}{z}H^*(X,\mathbb{C}[[Q]])[z^{-1}].
                              \end{align*} 
                              
                              The $J$-function of the cohomological GW theory is defined as
                               \begin{align*}
                              & J_H: \mathcal{H}_+ \to \mathcal{H}, \\
                                & J_H( \mathbf{t}(z))=-z+\mathbf{t}(z) +\sum_{d,n,a}\frac{Q^d}{n!}\varphi^a\la \frac{\varphi_a}{-z-\psi}, \mathbf{t}(\psi),\ldots , \mathbf{t}(\psi) \ra_{0,n+1,d} .
                               \end{align*}       
                       It is identified with the graph of the differential of the genus zero potential 
                               \begin{align*}
                               \mathcal{F}_0 (\mathbf{t}(z)) =\sum \frac{Q^d}{n!}\la \mathbf{t}(\psi),\ldots , \mathbf{t}(\psi) \ra_{0,n,d},
                               \end{align*}
                         viewed as a function of $\mathbf{t}(z)-z$ with respect to the polarization above. The  image of $J_H$ is a Lagrangian cone which we will denote $\mathcal{L}_H$. 
                        \begin{dilatoninput}
                        {\em The translation by $-z$ is called the } dilaton shift. {\em We will often refer to $\mathbf{t}(z)$ as} the input.  
                        \end{dilatoninput} 
                             
                       Twisted GW invariants are defined by considering in the integrals characteristic classes of push-forwards along the universal family of three types of tautological classes (they were called of type $\mathcal{A,B,C}$ in \cite{to1} ). Recall that the universal family of the moduli spaces $X_{0,n,d}$ can be identified with the map $\pi:X_{0,n+1,d}\to X_{0,n,d}$ which forgets the last marked point. The correlators of a twisted theory are typically cohomological integrals of the form 
                        
                       $$ \int_{[X_{0,n,d}]}\left( \prod_{m=1}^n ev_m^*(\varphi_m)\psi^{k_m}_m\prod_{i} \mathcal{A}_i(\pi_* (ev_{n+1}^* E)) \prod_j\mathcal{B}_j(\pi_* [F(L^\vee_{n+1})-F(1)])\prod_k\mathcal{C}_k(\pi_* i_*\mathcal{O}_\mathcal{Z}) \right),$$
                       where $\mathcal{A}_i, \mathcal{B}_j, \mathcal{C}_k$ are a finite number of multiplicative characteristic classes.
                        
                       One can similarly as above associate a Lagrangian cone to a twisted theory. The formalism of twisted GW theory (in genus zero, for the purpose of this paper) describes the correlators of a twisted theory in terms of the correlators of the untwisted theory. More precisely the three types of twistings and their effect on the correlators are :
                      \begin{itemize}
                       \item twistings by characteristic classes of index bundles $\pi_* (ev_{n+1}^* E)$. They correspond to rotation of the cone $\mathcal{L}_H$ by symplectomorphisms of $\mathcal{H}$ given by $End H^* (X)$ valued Laurent series in $z$. These symplectomorphisms are called loop group transformations.
                       \item kappa classes twistings by characteristic classes of $\pi_* [F(L^\vee_{n+1})-F(1)]$, where $F$ is a polynomial with values in $K^0(X)$. These correspond to a change of dilaton shift in the application point of the $J$-function. 
                       \item twistings by characteristic classes of $\pi_* i_*\mathcal{O}_\mathcal{Z} $, where $i:\mathcal{Z}\to X_{0,n+1,d}$ is the codimension two locus of nodes. 
                       These affect the generating series by a change of the space $\mathcal{H}_-$ of the polarization.
                      \end{itemize}      
                      
                     \begin{cohtwfake}
                     {\em Let us consider the twisted theory which we called fake in the previous section. It is given by inserting in the correlators the classes $\operatorname{Td}(T_{0,n,d})$}.
                     \end{cohtwfake}  
                          The virtual tangent bundle of $X_{0,n,d}$ can be written as a K-theoretic class (see \cite{tom}, Section 2.5)
                                  \begin{align*}
                                  T_{0,n,d} = \pi_*(ev^*_{n+1} T_X-1)   -\pi_*(L^{\vee}_{n+1}-1) -(\pi_*i_*\mathcal{O}_\mathcal{Z})^\vee.
                                  \end{align*}
                             We identify $\mathcal{K}^1$ and $\mathcal{H}$ via the Chern character
                            \begin{align*}
                            &\operatorname{qch}: \mathcal{K}^1 \to \mathcal{H},\\
                            & \Phi\mapsto \operatorname{ch(\Phi)},q\mapsto e^z.
                            \end{align*} 
                     \begin{cobcg}(\cite{tcag}, \cite{tom})\label{examplefake}
                    {\em The invariants of the fake theory are related to the cohomological GW invariants of $X$ by the following ingredients: }
                    \end{cobcg}
                  
                   \begin{itemize} 
                  \item the cone $\mathcal{L}_{fake}$  is given explicitly in terms of $\mathcal{L}_H$ by
                  \begin{align*}
                                           \operatorname{qch}(\mathcal{L}_{fake})=\triangle \mathcal{L}_H,
                                           \end{align*}
                 where the loop group transformation $\triangle $ is determined only by the characteristic class  $\operatorname{Td}$ and the first summand in the expression of $T_{0,n,d}$.
                  \item the change of dilaton shift in the application point of the $J$-functions of the theories from $-z$ to $\operatorname{qch}(1-q)$ is determined by the class $\operatorname{Td}$ and the second summand in $T_{0,n,d}$.
                  \item the generating series of the fake invariants is considered with respect to a different negative space on $\mathcal{H}$ determined  by the nodal class. More precisely 
                  the negative space $\mathcal{K}_-^1$ of the polarization is  spanned  by elements of the form $\{\Phi_a \frac{q^i}{(1-q)^{i+1}}\}_{i\geq 0}$. One way to  see this is by formally expanding  
                   \begin{align*}
                   \frac{1}{1-qL} = \sum_{i \geq 0} \frac{q^i}{(1-q)^{i+1}}(L-1)^i.
                   \end{align*}
                  The space $\mathcal{H}_-$ on the other hand is spanned by elements $\{\frac{\varphi_a}{z^i}\}_{i\geq 1}$. It is easy to see that the map $\operatorname{qch}$ does not identify them.
                   \end{itemize}
                    We refer the reader to \cite{to1} for explicit computations of this example as well as a treatment in full generality of the formalism of twisted GW theory.
                    
                    It will sometimes be convenient for us to write loop group operators as Euler-Maclaurin asymptotics of infinite products.  
                    \begin{mclaurin}
                    {\em Given a function $x\mapsto f(x)$, the Euler-Maclaurin asymptotics of the product $\prod_{r=1}^{\infty} e^{f(x-rz)}$ is obtained by writing}
                    \begin{align*}
                    \sum_{r=1}^{\infty}f(x-rz) & = (\sum_{r=1}^{\infty} e^{-rz\partial_x})f(x) = \frac{z\partial_x}{e^{z\partial_x}-1}(z\partial_x)^{-1}f(x)\\
                                               & = \frac{\int_0^x f(t)dt}{z}-\frac{f(x)}{2} +\sum_{k\geq 1}\frac{B_{2k}}{(2k)!}f^{(2k-1)}(x)z^{2k-1}. 
                    \end{align*}
                    \end{mclaurin}
                      
                   The operator $\triangle$ in Example \ref{examplefake} is the Euler-Maclaurin asymptotics of
                   \begin{align*}
                   \triangle\sim \prod_i \prod_{r=1}^{\infty} \frac{x_i}{1-e^{-x_i+rz}},
                   \end{align*} 
                  where $x_i$ are the Chern roots of the tangent bundle to $X$.
                  
                   \section{Twisted K-theoretic Gromov-Witten invariants}  \label{sec04}
                    We define twisted K-theoretic GW invariants by inserting in the correlators invertible multiplicative classes of index bundles $E_{n,d}:=\pi_*(ev_{n+1}^*E)$, where $E\in K^0(X)$. 
                    
                    The value of a  general K - theoretic invertible multiplicative class on a bundle $V$ is  
                    \begin{align}
                      \exp (\sum_{l}s_l \psi^l V). \label{multclass}
                      \end{align}
                   We will mainly work with $l<0$ summation range. We treat $s_l$ as formal parameters and expand the ground-ring of the theory 
                   by tensoring it with $\mathbb{C}[[s_1,s_2,\ldots]]$.  
                  
                  Hence the twisted invariants are defined by inserting in the correlators multiplicative classes of $E_{n,d}$
                     \begin{align*}
                   \la \mathbf{t}(L),\ldots ,\mathbf{t}(L)\ra_{0,n,d}^{tw}:= \chi\left(X_{0,n,d}; \mathcal{O}_{n,d}^{vir}\otimes_{i=1}^n \mathbf{t}(L_i)\otimes \exp (\sum_l s_l \psi^l E_{n,d})\right) .
                   \end{align*}
                   The twisted K-theoretic potential is defined as:
                   \begin{align*}
                   \mathcal{F}^{tw} =\sum_{d,n}\frac{Q^d}{n!}\la \mathbf{t}(L),\ldots ,\mathbf{t}(L)\ra_{0,n,d}^{tw} .
                   \end{align*}
                    
                   The J-function of the twisted theory is 
                   \begin{align*}
                  & \mathcal{J}^{tw}:\mathcal{K}_+ \to \mathcal{K}, \\
                  & \mathcal{J}^{tw}(\mathbf{t}(q))=1-q+\mathbf{t}(q) + \sum_{d,n,a}\frac{Q^d}{n!}\Phi^a\la \frac{\Phi_a}{1-qL}, \mathbf{t}(L),\ldots , \mathbf{t}(L) \ra_{0,n+1,d}^{tw} . 
                   \end{align*}
                   
                   \begin{euler1} \label{rescaling}
                   {\em  The bases $\{\Phi_a\}, \{\Phi^a\}$ involved in the definition of $\mathcal{J}^{tw}$ are dual with respect to the} twisted pairing {\em given by}
                   \begin{align*}
                   (\Phi_a, \Phi_b)=\chi(X, \Phi_a\otimes \Phi_b \otimes e^{\sum s_l \psi^l E}).
                   \end{align*}
                   {\em We will  have to consider various twisted theories. As a general rule the pairing  of a twisted theory is given by correlators on $X_{0,3,0}\simeq X$}
                    \begin{align*}
                     (\Phi_a, \Phi_b) = \langle \Phi_a,\Phi_b,1\rangle^{tw}_{0,3,0},
                     \end{align*}
                   {\em where the meaning of $\langle..\rangle^{tw}$ depends on the theory. To relate $J$-functions of different theories we need to regard them as elements of the same loop space. This involves rescaling the elements in loop spaces.}
                  
                   \end{euler1}
                  Let us define the cone $\mathcal{L}_{fake}^{tw}\subset \mathcal{K}^{1}$ the Lagrangian cone of the theory whose correlators are
                        \begin{align*}
                     \la \mathbf{t}(L),\ldots ,\mathbf{t}(L)\ra_{0,n,d}^{fake, tw} :=\int_{[X_{0,n,d}]} \prod_{i=1}^n\operatorname{ch}(\mathbf{t}(L_i))\operatorname{Td}(T_{0,n,d})\exp(\sum s_l \psi^l \operatorname{ch} E_{n,d}) ,
                        \end{align*} 
                        where $T_{0,n,d}$ is the virtual tangent bundle and the product is the cohomological cup product. Notice that the cone $\mathcal{L}_{fake}$ is the twisted fake cone at $s_l=0$. The $J$-function of the theory is 
                     \begin{align*}
                     J^{tw}_{fake}(\mathbf{t}(q))
                                        :=1-q+ \mathbf{t}(q)+ \sum_{d,n,a}\frac{Q^d}{n!}\Phi^a\la \frac{\Phi_a}{1-qL}, \mathbf{t}(L),\ldots , \mathbf{t}(L) \ra_{0,n+1,d}^{fake,tw}. 
                     \end{align*}
                     For now we restrict ourselves and make the following
                       
                        \begin{euler} \label{eu}
                         {\em The twisting class is  the K-theoretic Euler class  $e_K(E_{n,d})$. It is determined by its values on line bundles $e_K(L)=1-L^\vee$. To achieve this we sum after $l<0$ in the multiplicative class (\ref{multclass}) and set $s_l =- s_{-1}/l$.
                         We allow the twisting class to depend formally on one parameter $s_{-1}$ .  At $s_{-1}=-1$  ($\ref{multclass}$) becomes the Euler class.} 
                         \end{euler}
                       \begin{inv}
                       {\em In general $E_{n,d}$ can be written as the difference of two genuine bundles 
                                          $A_{n,d} \ominus B_{n,d}$ on $X_{0,n,d}$ (see \cite{coatesgiv}). We extend the definition of $e_K$ to such objects by working torus-equivariantly - where the action rotates the fibers of $E$. Then $e_K (E_{n,d}) = e_K(A_{n,d})e^{-1}_K(B_{n,d})$. }
                       \end{inv}  
                     
                     \begin{hypers}
                     {\em The case of the Euler class is the main motivation for considering twisted GW invariants: it can be used to relate GW of the ambient space $X$ with GW invariants of a subvariety given by the zero locus of a section of $E$.}
                     \end{hypers}   
                    The image of $\mathcal{J}^{tw}$ is a Lagrangian cone $\mathcal{L}^{tw}$. Our main result describes $\mathcal{L}^{tw}$ in terms of the cone $\mathcal{L}^{tw}_{fake}$.          
                    
                    \begin{positive}
                    {\em As the operators in the theorems are given as sums after $l\geq 1$ we adopt the convention $s_l=s_{-l}$ for all $l$, rather than writing $s_{-l}$ in all formulae.} 
                    \end{positive}
                    
                    \begin{main}\label{thm12}
                   
                    Let $\mathcal{J}_\eta^{tw}$ be the expansion in $(1-q\eta)$  of the twisted $J$-function. Then 
                   \begin{enumerate}
                    \item If $\eta$ is not a root of unity then $\mathcal{J}_\eta^{tw}$ is a power series in $(1-q\eta)$.
                    \item  $\mathcal{J}_1^{tw}$ lies on the cone 
                    \begin{align*}
                    \mathcal{L}_{fake}^{tw} = \exp (\sum_{l \geq 1} s_l \frac{\psi^l E^\vee}{1-q^l})\mathcal{L}_{fake}.
                    \end{align*} 
                   \item Assume $\eta\neq 1$ and that  Assumption  \ref{eu} holds. Let $\mathcal{T}^{tw}_k (\mathbf{f}^{tw})$ be as in  Definition \ref{tang1} but starting with a point $\mathbf{f}^{tw}\in\mathcal{L}_{fake}^{tw}$. 
                     Then
                   \begin{align*}
                   \mathcal{J}^{tw}_\eta (q^{1/k}\eta^{-1}) \in   \exp \sum_{i \geq 1} \left(\frac{\psi^i T_X^\vee}{i(1-\eta^{-i} q^{i/k})}-\frac{\psi^{ik} T_X^\vee}{i(1-q^{ik})} \right) R_\eta R_k^{-1}\mathcal{T}^{tw}_k (\mathcal{J}_1^{tw}(0)),
                   \end{align*} 
                   where $R_k,R_\eta$ are defined by 
                   \begin{align*}
                   & R_k:=\exp \left(\sum_{l\geq 1} s_{lk}\frac{k \psi^{lk} E^\vee}{1-q^{lk} }\right),\\
                   & R_\eta : = \exp\left(\sum_{l \geq 1}s_l \frac{\psi^l E^\vee}{1-q^{l/k} \eta^{-l}}\right) .
                   \end{align*}
                  \end{enumerate}
                     \end{main}
                  
                  \begin{rem}
                  {\em  As the first part of the theorem can be used to describe the tangent space at $0$ to $\mathcal{L}_{fake}^{tw}$ in terms of the cone $\mathcal{L}_{fake}$, Theorem \ref{thm12} gives a complete characterization of the twisted cone in terms of the untwisted one.}
                  \end{rem}
                  
                  \begin{rem11}\label{opform}
                  {\em The operators  $ \exp \left(\sum_{i \geq 1} \frac{\psi^i T_X^\vee}{i(1-\eta^{-i} q^{i/k})}-\sum_{i \geq 1} \frac{\psi^{ik} T_X^\vee}{i(1-q^{ik})} \right)$ and $R_\eta R_k^{-1}$  do not have poles at $q=1$.  For example $R_\eta$ has poles $\frac{s_{kl'}\psi^{kl'} E}{l'}$ at $q=1$ for $l=l'k$, they cancel the poles of $R_k$. Otherwise the last conditions of both Theorems \ref{mainadelic} and \ref{thm12} could not be true because modulo Novikov variables $\mathcal{J}_\eta$ is a power series.} 
                  \end{rem11}

                  \section{The proof of Theorem \ref{thm12}} \label{sec05}
                  We follow the proof of \cite{gito}. 
                  
                  The first condition in the theorem is obvious. For the second statement, let $\tilde{\mathbf{t}}(q) = $ contributions in the $J$-function of poles $\neq 1$. Then we claim that
                  \begin{fakest}
                  \begin{align*}
                  \mathcal{J}_1^{tw}(\mathbf{t}(q)) & = J^{tw}_{fake}(\mathbf{t}(q)+ \tilde{\mathbf{t}}(q))\\
                   & =1-q+ \mathbf{t}(q)+\tilde{\mathbf{t}}(q) + \sum_{d,n,a}\frac{Q^d}{n!}\Phi^a\la \frac{\Phi_a}{1-qL}, \mathbf{t}(L)+\tilde{\mathbf{t}}(L),\ldots , \mathbf{t}(L)+ \tilde{\mathbf{t}}(L) \ra_{0,n+1,d}^{fake,tw}. 
                  \end{align*}
                  \end{fakest}
                  \begin{proof}
                  This is completely analogous to the untwisted case: the contributions in the $J$-function with poles at $q=1$ correspond to Kawasaki strata  where the symmetries act trivially on the irreducible component - call it $C_+$ - carrying the distinguished marked point $x_1$. Such irreducible components can carry other special points - marked points or nodes. Let $p$ be such a node and call $C_-$ the irreducible component which intersects $C_+$ at $p$. The Euler class of the normal direction of the Kawasaki stratum  which smoothens the node $p$ is $(1-L_+L_-)$, where $L_+, L_-$ are cotangent line bundles at $p$ to the respective branches. The contribution in KRR coming from this normal direction is
                   \begin{align*}
                    \frac{1}{1-\operatorname{ch}(L_+)\operatorname{ch}(\operatorname{Tr}(L_-))}.
                   \end{align*}
                   Notice that the symmetry can not act with eigenvalue $1$ on $L_-$ otherwise we could smoothen the node while staying in the same Kawasaki stratum (equivalently the class in the denominator would be nilpotent).  
                   
                   Moreover the twisting class factorizes "nicely" over nodal strata, i.e. 
                  if $i$ is the inclusion of a divisor $X_{0,n_1+1,d_1}\times_X X_{0,n_2+1,d_2}$ parametrizing nodal curves in $X_{0,n,d}$ and $p_1, p_2$ the projections on the two factors the following holds (see \cite{coatesgiv}): 
                   \begin{align*}
                   i^* (\psi^l \pi_*(ev^* E))= p_1^* (\psi^l \pi_*(ev^* E)) +p_2^*(\psi^l\pi_*(ev^* E))- \psi^l ev_{node}^*E .
                   \end{align*}
                    The third summand is absorbed by the pairing at the node (see Remark \ref{rescaling}), the other two ensure that the twisting class distributes 
                    on the factors as twisting classes of the same form.
                     
                   This shows that the insertion in the correlators corresponding to the node $p$ comes from $\tilde{\mathbf{t}}(L_+)$. In fact when we sum after all possibilities of degrees and number of marked points of curves $C_-$ the insertion becomes  $\tilde{\mathbf{t}}(L_+)$. For a marked point on $C_+$ the insertion is $\mathbf{t}(L)$ hence the generating series $\mathcal{J}_1^{tw}(\mathbf{t}(q))$ is of the form 
                 \begin{align*} 
                  \mathcal{J}_1^{tw}(\mathbf{t}(q)) &=1-q+ \mathbf{t}(q)+\tilde{\mathbf{t}}(q) + \sum_{d,n,m,a}\frac{Q^d}{n!m!}\Phi^a\la \frac{\Phi_a}{1-qL}, \mathbf{t}(L),\ldots ,  \tilde{\mathbf{t}}(L) \ra_{0,n+m+1,d}^{fake,tw},
             \end{align*}
   where there are $n$ insertions of $\mathbf{t}(L)$  and $m$ insertions of  $\tilde{\mathbf{t}}(L)$ in the correlators. 
   Keeping in mind that there are $\binom{n+m}{m}$ ways of choosing the $n$ marked points among the $n+m$ special points we can rewrite  $\mathcal{J}_1^{tw}(\mathbf{t}(q))$ as
                    \begin{align*}
   \mathcal{J}_1^{tw}(\mathbf{t}(q)) = J^{tw}_{fake}(\mathbf{t}(q)+ \tilde{\mathbf{t}}(q)),
                  \end{align*}
 and hence $\mathcal{J}_1^{tw}$ lies on the fake twisted cone $\mathcal{L}^{tw}_{fake}$. 
                \end{proof}
                   The correlators $\langle ..\rangle^{fake,tw}$  are obtained from  $\langle ...\rangle^{fake}$ by inserting one more multiplicative characteristic class
                  \begin{align*}
                  \operatorname{ch}\left[\exp(\sum s_l \psi^l E_{n,d})\right].
                  \end{align*}
            This means $\mathcal{L}^{tw}_{fake}$ is obtained from $\mathcal{L}_{fake}$ by applying a loop group transformation, which we compute explicitly below.
            
            Let us extend the $\psi^l$ operations on cohomology using the Chern character.
                   It reads $\psi^l \varphi = l^j \varphi$ for $\varphi\in H^{2j}(X)$. Hence 
                   \begin{align}
                   \operatorname{ch}\left[\exp(\sum s_l \psi^l E_{n,d})\right] = \exp \left(\sum_{j\geq 0}(\sum_{l<0}s_l l^j) \operatorname{ch}_j E_{n,d}\right). \label{tw1234}
                   \end{align}
                  According to \cite{coatesgiv} the cone of a theory twisted by a general multiplicative characteristic class of the form 
                  \begin{align*}
                  \exp(\sum w_j \operatorname{ch}_j E_{n,d}) 
                  \end{align*} 
                  is obtained from the cone of the untwisted theory by applying the operator
                  \begin{align*}
                  \sum_{m,j\geq 0} w_{2m-1+j}\frac{B_{2m}}{(2m)!} \operatorname{ch}_j E \cdot z^{2m-1}.
                  \end{align*}
                 Here the Bernoulli numbers $B_{2m}$ are defined by 
                  \begin{align*}
                  \frac{t}{1-e^{-t}} = 1+ \frac{t}{2}+ \sum_{m\geq 1}\frac{B_{2m}}{(2m)!}t^{2m}. 
                  \end{align*}
                  We apply this to our twisting class $(\ref{tw1234})$ and we extract the coefficient of $s_l$ in the corresponding loop group transformation:
                   \begin{align*}
                  & \sum_{m,j\geq 0}l^{2m-1+j}\frac{B_{2m}}{(2m)!} \operatorname{ch}_j E \cdot z^{2m-1} =  \\
                  &  = \sum_{m,j\geq 0} l^j \operatorname{ch}_j E \frac{B_{2m}}{(2m)!}(lz)^{2m-1} = \\
                   &= \sum_{j\geq 0}\psi^l \operatorname{ch_j} E\left(\frac{lz}{lz(1-e^{-lz})} -\frac{1}{2}\right) =\\ 
                   & = \frac{\psi^l E}{1-e^{-lz}} - \frac{\psi^l E}{2}  = \frac{\psi^{-l} E^\vee}{1-q^{-l}} - \frac{\psi^{l}E}{2}.
                   \end{align*} 
                    The second summand is killed when we identify loop spaces (see Remark \ref{rescaling}). The first summand agrees  with the operator in part $(2)$ of the Theorem \ref{thm12}.
                   
                  We now proceed to prove part $(3)$ of Theorem \ref{thm12}: let $\eta$ be a primitive root of unity of order $k\neq 1$. The Kawasaki strata in
                  $X_{0,n,d}$ which give contributions with poles at $q= \eta^{-1}$ in the $J$-function were called {\em stem spaces} in \cite{gito}.
                They parametrize maps whose restriction  to the component $C_+$ carrying the first marked point factor through degree $k$ covers $z\mapsto z^k$. These maps can be identified with stable maps to the orbifold $X\times B\mathbb{Z}_k$ (of degree $k$ times less). The only points fixed by automorphisms of such maps are $0,\infty \in C_+$. However we can encounter $k$-tuples of nodes permuted by the symmetry. Let $(C_1,..,C_k)$ be the curves adjacent to these nodes: then the restriction of the stable map to $C_i$ have to be isomorphic and moreover $C_i$ are not allowed to carry marked points, as they have to be fixed by the symmetry.

           Hence the contributions in the $J$-function $\mathcal{J}$ with poles at $q=\eta^{-1}$ are cohomological integrals on the moduli spaces of maps to $X\times B\mathbb{Z}_k$ involving certain multiplicative characteristic classes coming from the tangent and normal directions to the Kawasaki strata and from the  index twisting (\ref{multclass}). It turns out these tangent and normal directions can be expressed in terms of 
       the universal families over the moduli spaces $(X\times B\mathbb{Z}_k)_{0,n,d}$, which we denote by $p$. Let $\mathbb{C}_{\eta^i}$ be the line bundle on  $X\times B\mathbb{Z}_k$ which is topologically trivial and on which $g$ acts as multiplication by $\eta^i$, for $i=0,1...k-1.$ 
       
       We define {\em the twisted stem} theory to be the cohomological GW theory of the target orbifold $X\times B\mathbb{Z}_k$ twisted by all the classes which contribute in the KRR formula applied to $\mathcal{J}^{tw}$.  We now list the classes:
                 \begin{itemize}
                 \item  the summand $\pi_*(ev_{n+1}^* T_X)$ of $T_{0,n,d}$ contributes the class
                                  \begin{align}
                                    \operatorname{Td}(p_* (ev^* T_X ))\prod_{i=1}^{k-1}\operatorname{Td}_{\eta^i}( p_*(ev^* (T_X\otimes \mathbb{C}_{\eta^i}))), \label{tloop}
                                    \end{align} 
                                    where
                                      \begin{align*}
                                     \operatorname{Td}_\lambda (L) =\frac{1}{1-\lambda e^{-c_1(L)}} . 
                                      \end{align*} 
                 \item  the summand $\operatorname{Td}(p_* (1-L_{n+1}^\vee))$ of $T_{0,n,d}$ contributes the class
                                      \begin{align}
                                     \operatorname{Td}(p_* (1-L^\vee))\prod_{i=1}^{k-1}\operatorname{Td}_{\eta^i}(p_*((1-L^{\vee})\otimes ev^*\mathbb{C}_{\eta^i})). \label{tdilaton} 
                                      \end{align}

                 \item the nodal contributions in KRR formula differ depending on the type of node. Denote by $\mathcal{Z}_g$ the nodes which can be smoothed within the same Kawasaki stratum and by $\mathcal{Z}_0$ the non-stacky nodes (these are disjoint from $\mathcal{Z}_g$). Then the nodal twisting is given by
                 \begin{align}
 \operatorname{Td}(-(p_*i_*\mathcal{O}_{Z_g})^\vee)\operatorname{Td}(-(p_*i_*\mathcal{O}_{Z_0})^\vee)\prod_{i=1}^{k-1}\operatorname{Td}_{\eta^i}(-(p_*i_*\mathcal{O}_{Z_0} \otimes ev^*\mathbb{C}_{\eta^i})^\vee).\label{tnodes}
 \end{align}
                 \item   the class (\ref{multclass}) contributes 
                            \begin{align}
                                    \operatorname{ch}\circ \operatorname{Tr}\left[\exp \left(\sum s_l \psi^l \pi_* ev^* E \right)\right].\label{arrtwist}
                                    \end{align}
                 
                 \end{itemize}
                 The first three types of twisting classes are present in \cite{gito}(Section $8$), where it is explained why they account for the tangent and normal directions to Kawasaki strata. We will express the class (\ref{arrtwist}) as a pushforward along $p$ in Proposition \ref{rcomp}.
      
       We denote the correlators of the twisted stem theory by $\langle .. \rangle^{stem,tw}$. The $J$ function of the theory is
                $$ \mathcal{J}^{st,tw}( \mathbf{t}(z))=-z+\mathbf{t}(z) +\sum_{d,n,a}\frac{Q^d}{n!}\varphi^a\la \frac{\varphi_a}{-z-\psi}, \mathbf{t}(\psi),\ldots , \mathbf{t}(\psi)\ra^{stem,tw}_{0,n+1,d} .$$
                
    Hence the polar part of $\mathcal{J}^{tw}_\eta(q)$ comes from correlators of the twisted stem theory. Let us denote by $\bar{\mathbf{t}}^{tw}(q)$ the contributions in the twisted $J$-function not having poles at $\eta^{-1}$. Then we claim that
                  \begin{stemtw}
                   \begin{align*}
                   \mathcal{J}^{tw}_\eta(q) = \bar{\mathbf{t}}^{tw}(q) +\sum_{n,d,a}\frac{Q^{dk}}{n!}\Phi_a\la\frac{\Phi^a}{1-q\eta L^{1/k}}, \mathbf{T}(L), \ldots , \mathbf{T}(L), \bar{\mathbf{t}}^{tw}(\eta^{-1}L^{1/k}) \ra^{stem, tw}_{0,n+2,d},
                   \end{align*}
                   {\em where the evaluation maps at the marked points land in components of $IB\mathbb{Z}_k$ labeled by the sequence $(g,1,..,1,g^{-1})$  and $\mathbf{T}(L) =\psi^k \widetilde{\mathbf{T}}(L)$, with   $\widetilde{\mathbf{T}}(q)=\mathcal{J}^{tw}_1(0)$}.
                  \end{stemtw}
                  In the above $\psi^k $ acts on cotangent line bundles $L\mapsto L^k$, elements of $K^0(X)$ and on Novikov variables $Q^d\mapsto Q^{kd}$.
                  \begin{proof}
 As already mentioned  the twisting class factorizes "correctly" over strata of  symmetries, hence the proposition is completely analogous to the same statement in \cite{gito}(Section $7$, Proposition $2$). We give a concise outline below.

Recall that stem spaces parametrize maps $C_+\to C \to X$, where the first map is $z\mapsto z^k$. The first 
                  and last seats in the correlators are $0,\infty\in C_+$ which are fixed by the symmetry. The insertion $\frac{1}{1-q\eta L^{1/k}}$ 
                 in the first seat occurs because $\operatorname{Tr}(L_1)=\eta L_1 $ and the cotangent lines on the cover and quotient curve differ by a power of $k$.  Summing after all possibilities  for $\infty$ ( it can be a node, a marked point or a non-special point of $C_+$) gives the insertion $\bar{\mathbf{t}}^{tw}(\eta^{-1}L^{1/k})$ for the last seat in the correlators.
                 
                       Let us explain the statement about the input $\mathbf{T}(L)$, to which we will refer as {\em the leg}: these are nodes on the quotient curve whose preimages on the cover are $k$-tuples of nodes connecting $C_+$ with curves $(C_1,...,C_k)$ which do not carry marked points . The maps $C_i \to X$, $i=1,..k$ are isomorphic.
                         Summing after all possibilities of degrees of maps   $C_i\to X$ we get contributions $\mathcal{J}^{tw}_1(0)$ for each such node.  Notice that on the cover curve there are $k$ copies of cotangent line bundles at the $k$ nodes to $C_i$, whereas on the quotient curve only one such cotangent line. Hence one needs to compute the trace of $\mathbb{Z}_k$ on the tensor product of the $k$ cotangent line bundles, where the generator $g\in \mathbb{Z}_k$ permutes the factors. 
                         The statement $\mathbf{T}(L) =\psi^k \widetilde{\mathbf{T}}(L)$ follows from the fact that for such an action of $\mathbb{Z}_k$ on the $k$-th power of a vector space $V$ we have $\operatorname{Tr}(g\vert V^{\otimes k}) = \psi^k V$  (see \cite{gito},Lemma in Section $7$).      
  \end{proof}

                       Hence
                     $\mathcal{J}^{tw}_\eta( \eta^{-1}q^{\frac{1}{k}})$ is obtained from a tangent vector to the cone of the twisted stem theory of $X\times B\mathbb{Z}_k$:
                     \begin{align}
                     \delta \mathcal{J}^{st,tw}(\delta \mathbf{t}, \mathbf{T}):=\delta \mathbf{t}(q^{1/k})+\sum_{n,d,a}\frac{Q^d}{n!}\la\frac{\Phi^a}{1- q^{1/k}L^{1/k}}, \mathbf{T}(L),....,\mathbf{T}(L),\delta \mathbf{t}(L^{1/k}) \ra^{stem,tw}_{0,n+1,d}, \label{st1}
                     \end{align} 
                     after changing $Q^d\mapsto Q^{dk}$ (but not in $\delta\mathbf{t}$).
             
              The Lagrangian cone of the cohomological GW theory of $X\times B\mathbb{Z}_k$ is the product of $k$ copies of Lagrangian cone of the GW theory of $X$. We will refer to each copy  as a {\em sector}. They are labeled by elements of $B\mathbb{Z}_k$ or equivalently connected components of $IB\mathbb{Z}_k$. 
              The tangent cone is accordingly a direct sum of $k$ copies of tangent spaces. Our tangent vector  $\delta \mathcal{J}^{st,tw}(\delta \mathbf{t}, \mathbf{T})$
              has application point in the sector labeled by $1$ of the cone but is tangent in the direction labeled by $g^{-1}$. 
               
                 To locate $\delta \mathcal{J}^{st,tw}(\delta \mathbf{t}, \mathbf{T})$ we process the classes involved in the twisted stem theory  according to the formalism of twisted cohomological GW theory of $X\times B\mathbb{Z}_k$.
                   \begin{itemize}
                   \item The class (\ref{tloop}) rotates the sectors labeled by $1$ , $g^{-1}$ by operators $\Box_k, \Box_\eta$. If $x_i$ are Chern roots of $T_X$ then they  are defined as asymptotics of the infinite products 
                    \begin{align*}
                    &\Box_k \sim \prod_i \prod_{r=1}^{\infty} \frac{x_i - rz}{1-e^{-kx_i+rkz}},\\
                    &\Box_\eta \sim \prod_i \prod_{r=1}^{\infty}\frac{x_i-rz}{1-\eta^{-r}e^{-x_i +rz/k}}.
                    \end{align*}
                    \item The class (\ref{tdilaton})
                       contributes to the change of dilaton shift which becomes $1-q^k$.
                      \item the nodal class (\ref{tnodes}) contributes a change of polarization. In the sector labeled by the identity the new polarization is given by expanding
                      \begin{align*}
                      \frac{1}{1-q^k L^k},
                      \end{align*}
                      whereas in the sector labeled by $g^{-1}$ is given by 
                      \begin{align*}
                      \frac{1}{1-q^{1/k}L^{1/k}}.
                      \end{align*}
                   \item  The index twisting class (\ref{arrtwist}) rotates the cone according to Proposition \ref{rcomp} below.
                  
                   \end{itemize}
                  \begin{identification}
                    {\em  The operator $\Box_\eta\Box_k^{-1}$ almost equals the operator in conditions 3 of Theorems \ref{mainadelic},\ref{thm12}. To see this (assume $x_i=x$) recall that $q=e^z$ and write}
                    \begin{align*}
                    ln (\Box_\eta\Box_k^{-1}) & = -\left(\sum_{r\geq 1} \operatorname{ln} (1-\eta^{-r}e^{-kx}e^{rz/k})\right)+\sum_{r\geq 1}\operatorname{ln}(1-e^{-kx}e^{rkz})\\
                     & = -\sum_{r\geq 1}\int \frac{\eta^{-r}e^{-x}e^{rz/k}}{1-\eta^{-r}e^{-x}e^{rz/k}} dx +\sum_{r\geq 1}\int \frac{k e^{-kx}e^{rkz}}{1 -e^{-kz}e^{rkz}}dx\\
                     & = -\sum_{r \geq 1}\int \left(\sum_{i\geq 1} \eta^{-ir}e^{-ix}e^{riz/k} \right)dx +\sum_{r\geq 1}\int \left(\sum_{i \geq 1}k e^{-ikx}e^{rikz}\right)dx\\
                     & = \int \sum_{i \geq 1} e^{-ix} \frac{\eta^{-i}e^{iz/k}}{1- \eta^{-i}e^{iz/k}} dx +\int \sum_{i \geq 1}k e^{-ikx} \frac{e^{ikz}}{1-e^{ikz}}dx\\
                     & = \sum_{i\geq 1} \frac{\psi^i T_X^\vee}{i(1-\eta^{-i}q^{i/k})} -\sum_{i\geq 1}\frac{\psi^{ik}T_X^\vee}{i(1-q^{ik})} +\operatorname{ln}\frac{1-T_X^\vee}{1-\psi^k T_X^\vee}.
                    \end{align*}
                    \end{identification}
                  
                  The constant factor $\frac{1-T_X^\vee}{1-\psi^k T_X^\vee}$ is absorbed by the change of pairing when identifying loop spaces as explained  in Remark \ref{rescaling}. We will ignore it from now on and slightly abusively write 
                   \begin{align*}
                   \Box_\eta\Box_k^{-1} = \exp\left(\sum_{i\geq 1} \frac{\psi^i T_X^\vee}{i(1-\eta^{-i}q^{i/k})} -\sum_{i\geq 1}\frac{\psi^{ik}T_X^\vee}{i(1-q^{ik})}\right).
                   \end{align*}
                  We now express  the class \eqref{arrtwist} in terms of the universal family $p$ and compute its effect on the cone.
                  
                   \begin{cohstem} \label{rcomp}
                  {\em  Twisting by the class \eqref{arrtwist} rotates the sector labeled by $g^{-1}$ of the Lagrangian cone of $X\times B\mathbb{Z}_k$ by }
                   \begin{align}
                   R_\eta : = \exp\left(\sum_{l \geq 1}s_l \frac{\psi^l E^\vee}{1-q^{l/k} \eta^{-l}}\right) . \label{looptrg}
                   \end{align}
                  {\em The sector labeled by the identity is rotated by}
                   \begin{align}
                   R_k:=\exp \left(\sum_{l\geq 1} s_{lk}\frac{k \psi^{lk} E^\vee}{1-q^{lk}}\right).
                   \end{align} 
                   \end{cohstem}
                  \begin{proof}
                   This is a computation based on Tseng's theorem \cite{tseng}.  
                    First we express the class (\ref{arrtwist}) in terms of the universal family $p$ over the moduli spaces $(X\times B\mathbb{Z}_k)_{0,n,d}$.  Recall that 
                   \begin{align}
                   \operatorname{Tr}_g(\pi_*(ev^* E)) = \sum_{i=0}^{k-1} \eta^{-i} p_*(ev^* E \otimes \mathbb{C}_{\eta^{i}}).\label{cohtw1}
                   \end{align}
                   Hence
                   \begin{align}
                    \operatorname{Tr}_g(\psi^l \pi_*(ev^* E)) = \sum_{i=0}^{k-1} \eta^{-il}\psi^l p_*(ev^* E \otimes \mathbb{C}_{\eta^{i}}).
                    \end{align}
                   Using the fact that $\psi^l \operatorname{ch}_j E = l^j \operatorname{ch}_j (E)$ we get 
                   \begin{align}
                   \operatorname{ch}\circ \operatorname{Tr}\left(\psi^l \pi(ev^*E)\right)& = \sum_{i=0}^{k-1} \eta^{-il} (\psi^l \operatorname{ch}p_*(ev^* E \otimes \mathbb{C}_{\eta^{i}}))=\nonumber\\
                   & =\sum_{i=0}^{k-1}\eta^{-il} \left(\sum_{j\geq 0} l^j \operatorname{ch}_j p_*(ev^* E \otimes \mathbb{C}_{\eta^{i}}))\right).
                   \end{align}
                   Therefore the contribution from \eqref{cohtw1} in the integrals equals:
                   \begin{align}
                   \prod_{i=0}^{k-1} \exp\left(\sum_{l\leq -1} s_l \eta^{-il}(\sum_{j\geq 0} l^j \operatorname{ch}_j p_*(ev^* E \otimes \mathbb{C}_{\eta^{i}}))\right).\label{cohtw2}
                   \end{align}
                   
                  This gives us a theory twisted by multiplicative characteristic classes of $k$ index bundles. 
                   
                   Let us recall Tseng's result on such twisted theories for the case of the target orbifold $X\times B\mathbb{Z}_k$. Its inertia orbifold consists of $k$ disjoint copies  $(X,g^i)$ for $i=0,1,..k-1$.  Consider a theory twisted by a characteristic class of the form
                   \begin{align*}
                   \exp\left(\sum_{j\geq 0} w_j \operatorname{ch}_j p_* (ev^* E) \right).
                   \end{align*}
                   
                   The Lagrangian cone defined by this theory is  obtained from the Lagrangian cone of the untwisted theory after multiplication by
                   \begin{align}
                   \exp\left(\sum_{j\geq 0}w_j \left(\sum_{m\geq 0}\frac{(A_m)_{j+1-m}z^{m-1}}{m!}+\frac{\operatorname{ch}_j E^{(0)}}{2}\right) \right).
                   \end{align}
                  
                  The operator  $A_m$ is defined by
                  \begin{align}
                  (A_m)_{\vert (X,g^i)}=\sum_{r=0}^{k-1}B_m(\frac{r}{k})\operatorname{ch}(E^{(r)}_{i}),
                  \end{align} 
                  where $E^{(r)}_{i}$ (respectively $E^{(0)})$ is the vector bundle over $(X, g^i)$ on which $g^i$ acts with eigenvalue $e^{2\pi i r/k}$(respectively $1$). $(A_m)_j$ is the degree $j$ piece of the operator $A_m$. The Bernoulli polynomials are defined by
                   \begin{align}
                   \sum_{m\geq 0}B_m(x)\frac{t^m}{m!}= \frac{te^{tx}}{e^t-1}.
                   \end{align}
                    Let us compute the symplectic transformation corresponding to the twisting \eqref{cohtw2} restricted to  $(X,g^{-1})$. We denote $\triangle_i$ the symplectic transformation corresponding to the contribution of $E\otimes \mathbb{C}_{\eta^{i}}$ in  the product in \eqref{cohtw2}. Then:
                   \begin{align}
                   \triangle_i = \exp \left(\sum_{j\geq 0}\left(\sum_{l\leq -1}s_l \eta^{-il}l^j\right)\left(\sum_{m\geq 0}\frac{B_m(i/k)\operatorname{ch}_{j+1-m}E}{m!}z^{m-1} \right) \right).
                   \end{align}
                  Let us extract the coefficient of $s_l$ in $R_\eta = \prod_i \triangle_i$. It equals:
                   \begin{align}
                   &\sum_{i=0}^{k-1}\sum_{j\geq 0}\eta^{-il}l^j\left(\sum_{m\geq 0}\frac{B_m(i/k)\operatorname{ch}_{j+1-m}E}{m!}z^{m-1} \right) =\nonumber\\
                   & =\sum_{i=0}^{k-1}\eta^{-il}\sum_{j\geq 0}\left(\sum_{m\geq 0}\frac{B_m(i/k)(l^m z^m) ( l^{j+1-m}\operatorname{ch}_{j+1-m}E)}{lz\cdot m!} \right) =\nonumber\\
                   & =\sum_{i=0}^{k-1}\eta^{-il}\sum_{s\geq 0}(l^s \operatorname{ch}_s E)\left(\sum_{m\geq 0}\frac{B_m(i/k)(l^m z^m)}{lz\cdot m!} \right)=\nonumber\\
                   &=\sum_{i=0}^{k-1}\eta^{-il}\sum_{s\geq 0}\psi^l \operatorname{ch}_{s}E \frac{zl e^{lzi/k}}{lz(e^{lz}-1)}=\nonumber\\
                   & =\psi^l \operatorname{ch}(E)\sum_{i=0}^{k-1}\eta^{-il} \frac{ e^{lzi/k}}{e^{lz}-1} =\nonumber\\
                   &= \psi^l \operatorname{ch}(E)  \frac{e^{lz}-1}{(e^{lz/k}\eta^{-l} -1)(e^{lz}-1)} = \frac{\psi^l \operatorname{ch}(E)}{e^{lz/k}\eta^{-l} -1}.
                   \end{align}
                   
                   Keeping in mind that $l<0$ we rewrite the result as
                   \begin{align*}
                   \frac{\psi^l \operatorname{ch}(E)}{e^{lz/k}\eta^{-l} -1} =  \frac{\psi^{-l} E^\vee}{1- q^{-l/k}\eta^{l}}-\psi^{l} E .
                   \end{align*}
                  The first term is the coefficient of $s_l$ in the answer stated in the theorem. The second terms give the correction
                    \begin{align*}
                    \exp\left( \sum s_l \psi^{l}E \right),
                    \end{align*} 
                    which is absorbed by the change of pairing on $\mathcal{K}$ as explained in Remark \ref{rescaling}.

                  For the sector corresponding to the identity the analogous computation with $B_m(i/k)$ replaced by $B_m(0)$ reveals the coefficient of $s_l$ to be $0$ if $k$ does not divide $l$ and 
                  \begin{align*}
                  \frac{ k\psi^{-l} E^\vee}{1-q^{-l}} 
                  \end{align*}  
                  if $k$ divides $l$.
                  \end{proof}
                  
                   Recall that $\mathcal{L}_H\subset \mathcal{H}$ is the cone of the (untwisted) cohomological GW theory of $X$. 
                   \begin{stems}\label{tansp}
                   $\operatorname{qch}\delta \mathcal{J}^{st,tw}(\delta \mathbf{t}, \mathbf{T})$ lies in the tangent space $ \Box_\eta R_\eta \Box_k^{-1}R_k^{-1} (\mathcal{T}_{\mathcal{I}^{tw}}\Box_k R_k \mathcal{L}_H )$ and the application point $\mathbf{T}$ is expressed in terms of $\mathcal{I}^{tw}$ by
                   \begin{align}
                   \operatorname{qch}
                   (1-q^k + \mathbf{T}(q)) =[\mathcal{I}^{tw}]_+. \label{polar1}
                   \end{align} 
                   \end{stems}
                   Here $[..]_+$ means projection along the negative space of the polarization of the sector labeled by $1$. 
                  \begin{proof} 
                   The series \eqref{st1} can be identified with a tangent vector to the cone of the twisted stem theory of $X/\mathbb{Z}_k$ in the sector labeled by $g^{-1}$. The application point belongs to the sector labeled by $1$, hence to the cone $\Box_k R_k \mathcal{L}_H.$ Since the $g^{-1}$-sector rotates by $\Box_\eta R_\eta$ the series belongs to the tangent space in the proposition. However the twisting by kappa classes and nodal classes in the twisted stem theory change the dilaton shift and the polarizations. The denominator $1-q^{1/k}L^{1/k}$ is equivalent to applying the polarization of the $g^{-1}$-sector to the same space. And the new dilaton shift is $1-q^k$, hence the relation between application points.
                   \end{proof}
                    
                     We are left with identifying the tangent space $\mathcal{T}_{\mathcal{I}^{tw}}\Box_k R_k \mathcal{L}_H$ with the $\mathcal{T}_k$ in the Theorem. We first show that  
                  \begin{psic} \label{psik}
                  {\em  Under the Assumption \ref{eu} the cone $\operatorname{qch}^{-1}(\Box_k\cdot R_k\mathcal{L}_H) = \psi^k \mathcal{L}^{tw}_{fake}.$}
                  \end{psic}
                  \begin{proof}
                   It is shown in \cite{gito}(Section $8$, Proposition $9$) that $\operatorname{qch}^{-1}(\Box_k \mathcal{L}_H) = \psi^k (\mathcal{L}_{fake})$. Since $\mathcal{L}^{tw}_{fake}= R_1 \cdot \mathcal{L}_{fake}$ and $R_k =\psi^k R_1$ under the Assumption \ref{eu}, the proposition follows.  
                  \end{proof}
                  \begin{11}
                  {\em Let $\mathcal{I}_{fake}$ be the point on $\mathcal{L}^{tw}_{fake}$ such that $\psi^k (\mathcal{I}_{fake}(\widetilde{\mathbf{T}})) = \mathcal{I}^{tw}(\mathbf{T})$. Then $\psi^k \widetilde{\mathbf{T}} =\mathbf{T}$.}
                  \end{11} 
                \begin{proof}
                  
                  Recall that $\mathcal{I}^{tw}$ is a point on the identity sector of the twisted stem theory: it lies on the cone  $\Box_k R_k \mathcal{L}_H$ , with the  
                  corresponding dilaton shift $1-q^k$ and polarization whose negative space is spanned by \{$\frac{q^{ki}}{(1-q^k)^{i+1}}\}_{i\geq 0} =\psi^k (\mathcal{K}_-^{1})$. Then
                   \begin{align*}
                    \mathcal{I}_{fake}(\widetilde{\mathbf{T}}) = (1-q) + \widetilde{\mathbf{T}} +\sum \frac{Q^d}{n!}\Phi_a\la  \frac{\Phi^a}{1-qL},\widetilde{\mathbf{T}}(L),\ldots ,\widetilde{\mathbf{T}}(L)\ra_{0,n+1,d}^{fake, tw} , \\
                   \mathcal{I}^{tw}(\mathbf{T}) = (1-q^k) + \mathbf{T} +\sum \frac{Q^d}{n!}\Phi_a\la\frac{\Phi^a}{1-q^kL^k},\mathbf{T}(L),\ldots ,\mathbf{T}(L)\ra^{st,tw}_{0,n+1,d} .
                     \end{align*}
                  
                   Now using $\psi^k (\mathcal{I}_{fake})=\mathcal{I}^{tw}$ it follows that  
                  $\mathbf{T}=\psi^k(\widetilde{\mathbf{T}})$. The constraints of the leg contributions in KRR impose that $\widetilde{\mathbf{T}}$ is  $\mathcal{J}_1^{tw}(0)$.
                  \end{proof} 
                   Moreover if we differentiate the relation $\psi^k (\mathcal{I}_{fake})=\mathcal{I}^{tw}$ it follows that
                   \begin{align*}
                   & \psi^k \left(\mathbf{f}(q) +\sum \frac{Q^d}{n!}\Phi_a\la \frac{\Phi^a}{1-qL},\widetilde{\mathbf{T}}(L),\ldots ,\widetilde{\mathbf{T}}(L),  \mathbf{f}(L)\ra_{0,n+2,d}^{fake, tw}\right) = \\
                   & \psi^k \mathbf{f}(q) + \sum \frac{Q^d}{n!}\Phi_a\la\frac{\Phi^a}{1-q^mL^m},\mathbf{T}(L),\ldots ,\mathbf{T}(L), \psi^k \mathbf{f}(L)\ra^{st,tw}_{0,n+2,d}.
                   \end{align*}
                   
                   On the RHS we have a point in the tangent space $\mathcal{T}_{\mathcal{I}^{tw}}\Box_k R_k \mathcal{L}_H$ (in the direction of $\psi^k \mathbf{f}(q))$. But on the LHS we have $\psi^k [S(q, Q) \mathbf{f}(q)] $ which almost belongs to $\mathcal{T}_k$ defined in Definition \ref{tang1}: we also need to 
                   change $Q^d\mapsto Q^{dk}$ in $S$ because the degrees in $\mathcal{J}_\eta^{tw}$ are multiplied by $k$.  
                   This concludes the proof of Theorem \ref{thm12}.
  
   \section{The  permutation-equivariant theory}\label{sec06}
    There is a natural $S_n$ action on the moduli spaces $X_{0,n,d}$ given by renumbering the marked points. Givental has recently generalized  the definition of K-theoretic GW invariants in this setting. He considers the $S_n$ modules
       \begin{align*}
       \left[\mathbf{t}(L),\ldots ,\mathbf{t}(L)\right]_{0,n,d}:= \sum (-1)^m H^m\left(X_{0,n,d}; \mathcal{O}_{n,d}^{vir}\otimes_{i=1}^n \mathbf{t}(L_i)\right) 
       \end{align*}
     where the input $\mathbf{t}(q)$ is a Laurent polynomial in $q$ with coefficients in $K(X)\otimes \Lambda$. Here $\Lambda$ is an algebra which carries 
     $\psi^k$ operations. Moreover for convergence purposes we assume $\Lambda$ has a maximal ideal $\Lambda_+$ and we endow it with the corresponding $\Lambda_+$-adic topology. The natural choices for $\Lambda$ satisfy these conditions - in general we want it to include the Novikov variables, the algebra of symmetric polynomials in a given number of variables and/or the torus equivariant $K$-ring of the point.
     For suitable choices of $\Lambda$ the permutation-quivariant invariants encode all the information about the $S_n$ modules above. We refer to \cite{giv} for details.
      
       The invariants :
      \begin{align*}
         \la \mathbf{t}(L),\ldots ,\mathbf{t}(L)) \ra^{S_n}_{0,n,d}
         \end{align*}
         are defined as K-theoretic push forwards of the classes  $\mathcal{O}_{n,d}^{vir}\otimes_{i=1}^n \mathbf{t}(L_i)$ along the map  $X_{0,n,d}/S_n \to [pt.]$. 
         
         One can define the $J$-function in the permutation-equivariant setting
               \begin{align*}
               \mathcal{J}_{S_\infty}(\mathbf{t}(q)):=1-q+\mathbf{t}(q) +\sum_{d,a} Q^d \Phi^a\la \frac{\Phi_a}{1-qL},\mathbf{t}(L),\ldots ,\mathbf{t}(L) \ra^{S_n}_{0,n+1,d}
               \end{align*}
            Givental noticed that  the combinatorics of the Kawasaki strata works the same as in the non permutation-equivariant theory. He used this to describe the Laurent expansion of $\mathcal{J}_{S_\infty}$  near each value of $q$. 
               \begin{snj} \label{adelthm}
               \em{ (\cite{giv}, Part III)} The values of $\mathcal{J}_{S_\infty}$ are characterized by:
         \begin{enumerate}
         \item $\mathcal{J}_{S_\infty}$ has poles only at roots of unity.\\
                       \item The expansion at $q=1$ $(\mathcal{J}_{S_\infty})_{(1)}$ lies on the cone $\mathcal{L}_{fake}$. \\
                 \item  $(\mathcal{J}_{S_\infty})_\eta (q^{1/k}\eta^{-1}) \in \Box_\eta \Box_k^{-1} \mathcal{T}_k(\mathcal{J}_{S_\infty}(\mathbf{t})_{(1)})$ , where $\mathcal{T}_k(\mathbf{f})$ is the space of  Definition \ref{tang1}. 
         \end{enumerate}
              
               \end{snj}
               Basically, one applies KRR and identifies the Laurent expansions of $\mathcal{J}_{S_\infty}$ with generating series of certain twisted theories as before. The only difference is that the legs are allowed to carry marked points, and condition $(3)$ of the theorem is modified accordingly.   
                  
                  We now proceed to define twisted permutation-equivariant K-theoretic GW invariants by tensoring the $S_n$ modules with multiplicative classes of $\pi_* E$: 
                  \begin{align*}
                   \la \mathbf{t}(L),\ldots ,\mathbf{t}(L)\ra_{0,n,d}^{S_n,tw}:= \chi\left(X_{0,n,d}/S_n; \mathcal{O}_{n,d}^{vir}\otimes_{i=1}^n \mathbf{t}(L_i)\otimes \exp (\sum_l s_l \psi^l E_{n,d})\right).
                   \end{align*}     
                 The $J$-function of the twisted permutation-equivariant quantum K-theory is 
                 
                 \begin{align*}
                       \mathcal{J}_{S_\infty}^{tw}(\mathbf{t}(q)):=1-q+\mathbf{t}(q) +\sum_{d,a} Q^d \Phi^a\la \frac{\Phi_a}{1-qL},\mathbf{t}(L),\ldots ,\mathbf{t}(L) \ra^{S_n, tw}_{0,n+1,d}.
                       \end{align*}
                 The characterization of the range of $\mathcal{J}^{tw}_{S_\infty}$ extends to this setup.
                       \begin{twsnj}  \label{thm13}
                      The values of $\mathcal{J}^{tw}_{S_\infty}$ are characterized by:
                     \begin{enumerate}
                       \item $\mathcal{J}^{tw}_{S_\infty}$ has poles only at roots of unity.\\
                    \item The expansion at $q=1$ $(\mathcal{J}^{tw}_{S_\infty})_{(1)}$ lies on the cone $\mathcal{L}^{tw}_{fake}$. \\
               \item $(\mathcal{J}^{tw}_{S_\infty})_\eta (q^{1/k}\eta^{-1}) \in R_\eta R_k^{-1}\Box_\eta \Box_k^{-1} \mathcal{T}_k(\mathcal{J}^{tw}_{S_\infty}(\mathbf{t})_{(1)})$ , where $\mathcal{T}^{tw}_k(\mathbf{f}^{tw})$ is given by the procedure described in Definition \ref{tang1}, but starting with the point $\mathbf{f}^{tw}\in \mathcal{L}_{fake}^{tw}$.
                     \end{enumerate}  
                     
                       \end{twsnj} 
                       
                       \begin{proof} 
                         Again  the main difference with the non permutation-equivariant case is that  we do not impose the condition $\mathbf{t}(q)=0$ on the definition of $\mathbf{T}(q)$ because we are allowed to permute marked points. Hence the space $\mathcal{T}^{tw}_k$ in condition $(3)$ is obtained from the tangent space to $\mathcal{L}^{tw}_{fake}$ at $\mathcal{J}^{tw}_{S_\infty}(\mathbf{t})_{(1)}$.    
                           \end{proof}
                          
                           Remarkably, from the two {\em local} characterizations above we obtain  a {\em global} relation, albeit under the restrictions of the Assumption \ref{eu} 
                          \begin{globaltwsnj} \label{snthm}
                          Assume the characteristic class in the permutation-equivariant twisted theory is the Euler class. Let $\mathcal{L}_{S_\infty}$ and $\mathcal{L}_{S_\infty}^{tw}$ denote the ranges of  the $J$-functions  $\mathcal{J}_{S_\infty}$ and $\mathcal{J}_{S_\infty}^{tw}$ respectively. Then
                          \end{globaltwsnj}
                           \begin{align*}
                           \mathcal{L}_{S_\infty}^{tw}= e^{\sum_{l>0}s_l \frac{\psi^l E^\vee}{(1-q^l)}}\mathcal{L}_{S_\infty}.
                           \end{align*} 
                           
                           \begin{proof} Let 
                             
                            \begin{align*}
                             \mathbf{g}(q) = e^{\sum_{l>0}s_l \frac{\psi^l E^\vee}{(1-q^l)}} \mathbf{f}(q),
                            \end{align*}
                            where $\mathbf{f} \in \mathcal{L}_{S_\infty}$. We will prove that $\mathbf{g}$ satisfies the conditions in Theorem \ref{thm13} assuming $\mathbf{f}$ satisfies 
                            the conditions of Theorem \ref{adelthm}. The first one is obvious. 
                            
                              For the second condition notice that 
                              \begin{align*}
                                \mathbf{g}_{(1)} = e^{\sum_{l>0}s_l \frac{\psi^l E^\vee}{(1-q^l)}} \mathbf{f}_{(1)}.
                               \end{align*}
                             
                             Since we assume by Theorem \ref{adelthm} that $\mathbf{f}_{(1)} \in \mathcal{L}_{fake}$ and we proved that 
                             \begin{align*}
                             \mathcal{L}_{fake}^{tw} = e^{\sum_{l>0}s_l \frac{\psi^l E^\vee}{(1-q^l)}} \mathcal{L}_{fake},
                              \end{align*} 
                             it follows that $\mathbf{g}_{(1)}\in \mathcal{L}_{fake}^{tw}$.
                             
                              Also notice that if the tangent space at $\mathbf{f}_{(1)}$ to $\mathcal{L}_{fake}$ is given as the image of a map 
                                \begin{align*}
                                S(q,Q):\mathcal{K}_+ \to \mathcal{K}, 
                                \end{align*}
                              then the same tangent space at $\mathbf{g}_{(1)}$ to $\mathcal{L}_{fake}^{tw}$ is given by 
                              \begin{align*}
                              S'(q,Q) = e^{\sum_{l>0} s_l \frac{\psi^l E^\vee}{1-q^l}} S(q,Q):\mathcal{K}_+ \to \mathcal{K}.
                              \end{align*}  
                              
                              According to our assumption
                                 \begin{align*}
                                 \mathbf{f}(q^{1/k}\eta^{-1})\in \Box_\eta \Box_{k}^{-1} S(q^k, Q^k)\mathcal{K}^{fake}_+.
                                 \end{align*}
                             It is an easy computation to see that
                             \begin{align*}
                                   \mathbf{g}(q^{1/k}\eta^{-1}) &=e^{\sum_{l>0}s_l \frac{\psi^l E^\vee}{(1-q^{l/k}\eta^{-l})}}\mathbf{f}(q^{1/k}\eta^{-1}) = \\
                                                                & = R_\eta \mathbf{f}(q^{1/k}\eta^{-1})
                                    \in R_\eta R_k^{-1}\Box_\eta \Box_{k}^{-1} S'(q^k, Q^k)\mathcal{K}^{fake}_+
                                   \end{align*}
                                  if all $s_l=- 1/l$. This concludes the proof.
                               \end{proof}
                               
              \begin{shift}
                 {\em In the non permutation-equivariant case it was difficult to express the application point of the twisted $J$ function in terms of $\mathcal{J}(\mathbf{t}(q))$. In the  permutation-equivariant case , Theorem \ref{snthm} above allows us to achieve this very nicely. More precisely the projection to $\mathcal{K}_+$ of an element }
                   \begin{align*}
                   e^{\sum_{l>0} \frac{\psi^l E^\vee}{l(q^l-1)}}\mathcal{J}_{S_\infty}(\mathbf{t}(q))
                   \end{align*} 
                  {\em is}  $1-q +\mathbf{t}(q) - E^\vee$.
                  \end{shift}
               
              As a consequence of Theorem \ref{snthm} we can describe the cone of a theory  twisted by a general multiplicative class. Define a twisted theory by inserting in the correlators the general multiplicative class
                \begin{align*}
                \exp\left(\sum_{l<0} s_l \psi^l E_{n,d}\right)
                \end{align*} 
                and assume for convergence purposes that the class $E\in K^0(X,\Lambda_+)$ ( and $\psi^l$ acts on the coefficient in $\Lambda_+$). Denote by $\mathcal{L}_{S_\infty}^{tw}$ the range of its $J$-function. Then
                \begin{general}
               \begin{align*}
               \mathcal{L}_{S_\infty}^{tw} =  e^{\sum_{l} s_l\frac{\psi^l E^\vee}{(1-q^l)}}\mathcal{L}_{S_\infty}.
               \end{align*}
                \end{general}
             \begin{proof}
             
             We want to express the multiplicative class as a linear combination of Euler classes of $\psi^k E$:
             \begin{align*}
             \exp\left(\sum_{l<0} s_l \psi^l E_{n,d}\right)=\prod_{k\geq 1}(e_K(\psi^k E))^{t_k}
             \end{align*}
             
             This gives the system
             \begin{align*}
             \sum_l s_l \psi^l V  =\sum_{k>0} t_k \left[\sum_{i<0}\frac{\psi^{ki}V}{i}  \right],
             \end{align*}
or equivalently 
              \begin{align*}
              ls_l =\sum_{k\vert l, k>0}kt_k,\qquad l=-1,-2,...
              \end{align*}
             It can be solved by M\"obius inversion formula.  Using then Theorem \ref{snthm}  concludes the proof.
             \end{proof}
             
             Let us now recall the $\mathcal{D}_q$ module structure recently proved in \cite{giv}. For a Novikov variable $Q_i$ let $p_i\in H^2(X)$ the dual cohomological class and let $P_i=e^{-p_i}\in K^0(X)$. It is known (\cite{gito}) that in the non-permutation equivariant case the operator $P_i q^{Q_i\partial_{Q_i}}$  preserves tangent spaces to the cone $\mathcal{L}$. The analogue statement in the permutation-equivariant theory is
             the following
            \begin{dmodule}   \label{dmod}
            {\em (\cite{giv}, Part IV) Let $\lambda\in \Lambda_+$. Then the cone $\mathcal{L}_{S_\infty}$ is invariant under expressions of the form} 
             \begin{align*}
     \exp\left( \sum_{k >0} \frac{\psi^k (\lambda D(P_i q^{Q_i\partial_{Q_i}},q))}{k(1-q^k)}\right),
           \end{align*}  
   {\em where $D$ is a Laurent polynomial in $P_i q^{Q_i\partial_{Q_i}},q$ with coefficients from $\Lambda$ independent of $Q$}.    
           \end{dmodule}
   
   We combine Theorem \ref{dmod} with Theorem \ref{snthm} to prove  a "quantum Lefschetz" general result.
    \begin{lefschetz}\label{qlefschetz}
    Let $V\subset X$ be a hypersurface given as the zero section of a convex line bundle $L$. Let\footnote{We discard $S_\infty$ from the notation as
    we will only talk about permutation-equivariant theory from now on.} 
    \begin{align*}
\mathcal{J}_X = \sum_{d\in Eff(X)} J_d Q^d    
    \end{align*}
 be a point on the cone of the permutation equivariant theory of $X$. Then the point
   \begin{align*}
              \mathcal{I}_V=\sum_{d\in Eff(X)}J_d Q^d \prod_{r=1}^{\langle c_1(L),d\rangle}(1- L^\vee q^r) 
              \end{align*} 
   lies on the cone of the permutation-equivariant K-theory of $V$. 
    \end{lefschetz} 
    More precisely  if $i:V\to X$ is the inclusion then
      \begin{align*}
      e_K(L) \mathcal{I}_V = i_*\mathcal{J}_V(i^*\mathbf{t}(q)), 
      \end{align*}
        where $\mathbf{t}(q)$ can be explicitly computed via projection to $\mathcal{K}_+$ and $i_*$ on the RHS acts also on the Novikov variables  via the natural map $i_*: H_2(V)\to H_2(X)$.
    
    \begin{proof} 

 The arguments of \cite{kkp} extend in K-theory to show that 
 
 \begin{align*}
 \mathcal{O}^{vir}_{n,d, V} = e_K(L_{n,d})\otimes \mathcal{O}^{vir}_{n,d,X}.
 \end{align*}
 Hence  the K-theoretic GW theory of $X$ twisted by the Euler class of $L_{n,d}$ gives the K-theoretic GW theory of $V$.
 
 Let us write $L$ as a monomial $f(P_i^{-1})$. 
  Let \begin{align*}
  \Gamma_q(x)= e^{\sum_{k>0} \frac{x^k}{k(1-q^k)}}\sim \prod_{r=0}^{\infty}\frac{1}{1-xq^r}.
  \end{align*}
 Then the operator 
 \begin{align}
 \frac{\Gamma_{q^{-1}}(f(P_i))}{\Gamma_{q^{-1}}(f(P_i q^{Q_i\partial_{Q_i}}))}\label{op1}
 \end{align} 
 acts as 
 \begin{align*}
 Q^d\mapsto Q^d \prod_{r=1}^{\langle c_1(L),d\rangle}(1-L^\vee q^r),
 \end{align*} 
hence we get 
\begin{align*}
 e_K(L)\mathcal{I}_V = e^{\sum_{k>0}\frac{\psi^k L^\vee}{k(q^k-1)}}\frac{1}{\Gamma_{q^{-1}}(f(P_iq^{Q_i\partial_{Q_i}}))} \mathcal{J}_X.
\end{align*}
According to Theorem \ref{dmod} the operator in the denominator preserves the cone of the untwisted theory of $X$. The other operator 
on the RHS moves the point on the cone of the theory twisted by $e_K(L_{n,d})$. The claim follows. 
      \end{proof}  
  In particular for $X=\mathbb{CP}^N$ we confirm results of \cite{giv}, where the following was proved using localization    
       \begin{projhyper}
       {\em (\cite{giv}, Part V) Let $V\subset \mathbb{CP}^N$ be a hypersurface given as the zero section of $\mathcal{O}(l)$ for some $l > 0$. Then}
      \begin{align*}
                           \mathcal{I}_V=(1-q)\sum_{d\geq 0} Q^d\frac{\prod_{r=1}^{dl}(1- P^l q^r) }{\prod_{r=1}^d(1- P q^r)^{N+1}}
                          \end{align*} 
         {\em is a point on the cone of the permutation equivariant K-theory of $V$.}                 
       \end{projhyper}
   \begin{proof}    
     The $J$-function of $\mathbb{CP}^N$  at $\mathbf{t}(q)=0$ is known (\cite{giv_lee}) to be 
               \begin{align*}
               \mathcal{J}_{\mathbb{CP}^N}(0)=(1-q)\sum_{d\geq 0}\frac{Q^d }{\prod_{r=1}^d(1-Pq^r)^{N+1}}.
               \end{align*} 
    Applying Theorem \ref{qlefschetz} to this  series gives the result.             
      \end{proof}                                      
    \begin{cigeneral}
    {\em Theorem \ref{qlefschetz} has a straight forward generalization for complete intersections given as zero sections of direct sums of convex line bundles on $X$.}
    \end{cigeneral}  
    
  Another application of our Theorem \ref{snthm} is to find points on the cone of the total space $E$ of a toric fibration $E\to B$ given a point 
  on the base $B$ of the fibration. The proof is along the lines of \cite{jeff} where it was done in cohomological GW theory. 
  
  First let us introduce notation. Let $X$ be a toric non-singular compact K\"ahler manifold. It can be described by symplectic reduction. Let 
  the torus $T^N$ act on $\mathbb{C}^N$ endowed with the canonical symplectic form in the usual way. The moment map of this action is
   $\mu:\mathbb{C}^N\to \mathbb{R}^N, \mu (z_1,...,z_N) = (\vert z_1\vert^2,...,\vert z_N\vert^2)$. For the action of a  subtorus $T^K \subset T^N$ the moment map is obtained as the composition $\mathbf{m}\circ \mu : \mathbb{C}^N \to \mathbb{R}^k$, where $\mathbf{m}:\mathbb{R}^N\to \mathbb{R}^K$ is the dual of the embedding of Lie algebras $\operatorname{Lie}(T^K)\subset \operatorname{Lie}(T^N)$.
   We denote the elements of the matrix $\mathbf{m}$ by $m_{ij}$. Applying symplectic reduction over a regular value 
   $\omega$ of the moment map we get a toric variety $X =\mathbb{C}^N//_\omega T^K$ of dimension $N-K$.
   
   The fibration $(\mathbf{m}\circ\mu)^{-1}(\omega)\to X$ endows $X$ with $K$ tautological line bundles which we denote $P_i$. They represent a basis of 
   $Pic(X)$ and generate $K^0(X)$. 
   
   Let $B$ be K\"ahler manifold, $L_i$ line bundles on $B$, $i=1,..,N$. We replace the fiber of $\oplus L_i$ with the toric manifold $X$, obtaining this way a toric fibration $\pi:E\to B$. It carries a fiberwise action of $T^N$. The total space $E$ carries $K$ tautological line bundles $\mathcal{P}_i$ which restrict to $P_i$ on each fiber. They generate $K^0(E)$ as an algebra over $K^0(B)$. 
   
   Similarly a degree $\mathcal{D}\in H_2(E,\mathbb{Z})$ "breaks up" as a degree $D =\pi_*(\mathcal{D})\in H_2(B,\mathbb{Z})$  and degrees 
   $d_i:=-\langle c_1(\mathcal{P}_i),\mathcal{D}\rangle$ along the fibers. We will denote the two sets of Novikov variables by $Q_B, Q$ i.e. 
   $Q_B^D$ represents $D$ in the Mori cone of $B$ and $Q^d =Q_1^{d_1}\cdot..\cdot Q_K^{d_K}$. Let us define for $j=1,...N$
   \begin{align*}
   U_j(\mathcal{P}) = \prod_{i=1}^K \mathcal{P}_i^{m_{ij}}L_j^\vee, \qquad U_j(\mathcal{D}) = \sum_{i=1}^K d_i m_{ij} + \langle c_1(L_j), D\rangle .
   \end{align*}
  We can now state
  \begin{toricfibration} \label{tfibration}
  Let 
  \begin{align*}
  \mathcal{J}_B(\mathbf{t}(q))=\sum_{D\in Eff(B)} J_d Q_B^D 
  \end{align*}
be a point on the Lagrangian cone of the permutation-equivariant K-theory of $B$. Then 
\begin{align*}
\mathcal{I}_E := \sum_{d\in \mathbb{Z}^K,D\in Eff(B)} J_d Q_B^D Q^d \prod_{j=1}^N \frac{\prod_{r=-\infty}^{0}(1-U_j(\mathcal{P}) q^r)}{\prod_{r=-\infty}^{U_j(\mathcal{D})}(1-U_j(\mathcal{P}) q^r)}
\end{align*}  
lies on the cone of the total space $E$.
  \end{toricfibration}
 \begin{proof}
We use localization along the fibers. Most of the details are common with \cite{jeff} - where it was carried in the cohomological theory and \cite{giv} -where it was done for the case $B=\operatorname{pt.}$ 

Let us denote by $\mathbb{C}[\Lambda_1^{\pm 1}, \ldots , \Lambda_N^{\pm 1}]$ the ring $K^0(BT^N)$. We will work torus-quivariantly and deduce the statement of the theorem as the limit $\Lambda_i\to 1$.
Let us label the fixed points of the torus action on $X$  by multiindexes $\alpha=(j_1,...,j_K)$ which specify $K$-dimensional faces of the first orthant whose image under the map $\mathbf{m}$ contains $\omega$. Toric one dimensional orbits connecting the fixed points $\alpha$ and $\beta$ exist precisely when $\alpha \cup\beta$ has cardinality $K+1$. For a fibration $\pi: E\to B$ with fibers isomorphic to $X$ fixed points of the fiberwise action of $T^N$ form sections
$\alpha:B\to  E$, one  for each fixed point $\alpha \in X$.  The normal bundle of the section $\alpha$ is the sum of $N-K$ line bundles
  \begin{align*}
   U_j(P^\alpha):=\alpha^* U_j = \otimes_i (P_i^\alpha)^{m_{ij}}L^\vee_j \Lambda_j^{-1}, \quad j\notin \alpha 
  \end{align*}
where $P_i^\alpha$ are determined by 
\begin{align*}
\otimes_i (P_i^\alpha)^{m_{ij}}L^\vee_j =\Lambda_j,\quad j\in \alpha.
\end{align*}
Let $\mathbf{f}$ be a point on the cone $\mathcal{L}_E$ of permutation-equivariant quantum K-theory of $E$. We denote by 
\begin{align*}
\mathbf{f}^\alpha : = \alpha^* \mathbf{f} 
\end{align*}
its restriction to the fixed point section $\alpha$. Then localization gives
 \begin{align*}
 \mathbf{f} =\sum_\alpha \alpha_*\left( \frac{\mathbf{f}^\alpha}{\prod_{j\notin \alpha} e^T_K (U_j(P^\alpha))}\right),
 \end{align*}
 
 where $e_K^T$ is the torus equivariant K-theoretic Euler class.
 
 The characterization of points on the cones of toric varieties using localization given in \cite{giv} extends to this setting in the following way:
\begin{last}\label{cond}
{\em The point $\{\mathbf{f}^\alpha\}$ belongs to the cone $\mathcal{L}_E$ iff the following are satisfied
\begin{enumerate}
\item As a meromorphic function near the roots of unity $\mathbf{f}^\alpha \in \mathcal{L}^\alpha$, where $\mathcal{L}^\alpha$ is the cone of $\alpha(B)$ twisted by the inverse of the Euler class of the normal bundle $(e_K^T)^{-1}(N_\alpha)$. The variables $\Lambda_i^{\pm 1}$ and the Novikov variables are considered as elements of the coefficient ring.\\
\item The other poles, which are simple for generic values of $\Lambda_i$, come from factors of the form $(1-q^m U_j(P^\alpha))$  for $j\notin \alpha$. They have residues controlled recursively in degrees by
\begin{align*}
\operatorname{Res}_{q=\lambda^{1/m}}\mathbf{f}^\alpha (q)\frac{dq}{q}=\frac{Q^{md_{\alpha\beta}}}{e_K^T(\mathcal{N}_{\alpha\beta}^m) }\mathbf{f}^\beta(\lambda^{1/m}),
\end{align*}
where $\lambda=U_j(P^\alpha)$, $\beta$ is determined by $j$, $d_{\alpha\beta}$ is the degree of the one-dimensional orbit $\simeq \mathbb{CP}^1$ connecting $\alpha$ with $\beta$
and $\mathcal{N}_{\alpha\beta}^m$ is the normal bundle to the moduli spaces of maps to the orbit of degree $kd_{\alpha\beta}$ with two marked points at the point which is the degree $m$ cover of the orbit.
\end{enumerate}
}
 
\end{last} 

We need to check that the equivariant version of $\mathcal{I}_E$  satisfies the two conditions in Proposition \ref{cond}. First notice that 
\begin{align*}
\mathcal{I}_E^\alpha = \sum_{D,d} \frac{J_D Q_B^D Q^d}{\prod_{j\in \alpha}\prod_{r=1}^{U_j(\mathcal{D})}(1-q^r)\prod_{j\notin \alpha}\prod_{r=1}^{U_j(\mathcal{D})}(1-q^rU_j(P^\alpha))}.
\end{align*}  
The second condition is verified by the computation in \cite{giv} which carries over without any modifications.  To verify the first condition
introduce monomials in the Novikov variables $Q_j^\alpha$ such that $\prod_{j\in \alpha}Q_j^{\alpha}=Q^d$. For  $j\notin \alpha$ introduce monomials $Q_{B,j}^\alpha$ dual to the cohomology class $-c_1(U_j(P ^\alpha))$. 
Notice that the operator 
 \begin{align*}
 \prod_{j\in \alpha}\Gamma_q (Q_j^\alpha)  \prod_{j\notin \alpha} \frac{\Gamma_{q^{-1}}(U_j(P^\alpha)q^{Q_{B,j}^\alpha \partial_{Q_{B,j}^\alpha}} )}{\Gamma_{q^{-1}}(U_j(P^\alpha))} 
 \end{align*}
 transforms $\mathcal{J}_B$ into $\mathcal{I}_E$. According to the $\mathcal{D}_q$ module structure the first factors and the numerators maintain $\mathcal{J}_B$ on the cone of the untwisted theory of $\alpha(B)$, according to Theorem \ref{snthm} the denominators move points on the cone $\mathcal{L}^\alpha$ of the theory twisted by $e_K^T(N_\alpha)^{-1}$. This concludes the proof.
\end{proof} 

In the end we illustrate some computations how one can use our results to compute K-theoretic Gromov-Witten invariants of the complete intersections and toric fibrations of the Theorems \ref{qlefschetz} and \ref{tfibration} starting from their $\mathcal{I}$-functions. Let
$X\subset \mathbb{CP}^4$ be a hypersurface given as the zero section of $\mathcal{O}(5)$. Then we have proved that the hypergeometric series
      \begin{align*}
                           \mathcal{I}_X=(1-q)\sum_{d\geq 0} Q^d\frac{\prod_{r=1}^{5d}(1- P^5 q^r) }{\prod_{r=1}^d(1- P q^r)^{5}}
                          \end{align*} 
 lies on the cone $\mathcal{L}$ of the permutation-equivariant quantum K-theory of $X$. Write the coefficient of $Q$ as 
 $$Q\frac{(1-q)\prod_{r=1}^5(1-P^5q^r)}{(1-Pq)^5} := Q\cdot f(P,q) + \frac{Q \cdot g(P,q)}{(1-Pq)^5},$$
  where $f(P,q)$ is a polynomial in $q$ (hence contributes to $\mathbf{t}(q)$) and the fraction $\frac{Q \cdot g(P,q)}{(1-Pq)^5}$ belongs to $\mathcal{K}_-$ and hence comes from correlators. An immediate degree argument shows that we must have
  $$\frac{ g(P,q)}{(1-Pq)^5} =\sum_a \Phi^a \langle \frac{\Phi_a}{1-qL}\rangle_{0,1,1}^X . $$
  Pairing this expression against other classes using the K-theoretic Poincare pairing on $X$ we get all one point degree one invariants:
  $$ \langle \frac{\Phi_a}{1-qL}\rangle_{0,1,1}^X= \left(\frac{ g(P,q)}{(1-Pq)^5}, \Phi_a\right)_X  = -\operatorname{Res}_{P=1}\frac{g(P,q)\Phi_a (1-P^5)}{(1-P)^5(1-Pq)^5}\frac{dP}{P}.$$
  
  Taking $\Phi_a=1$ above we get the one point invariants 
  
  $$\langle \frac{1}{1-qL}\rangle_{0,1,1}^X = \frac{2875(1-3q)}{(1-q)^2}.$$ 
  Notice that setting $q=0$ one computes the invariant $\langle 1 \rangle_{0,1,1}^X = 2875$. This is unsurprising as according to the K-theoretic string equation  $\langle 1 \rangle_{0,1,1}^X $ equals the number of lines in $X$.
  
  Our results hold independent of the degrees of the  equations cutting out the complete intersections in projective space. For another example, consider $Y\subset \mathbb{P}^5$ given as the intersection of two quadric hypersurfaces. Then the same computation as above starting from the $\mathcal{I}$-function 
  \begin{align*}
  \mathcal{I}_Y=(1-q)\sum_{d\geq 0} Q^d\frac{\prod_{r=1}^{2d}(1- P^2 q^r)^2 }{\prod_{r=1}^d(1- P q^r)^{6}},
  \end{align*}
  gives the one point invariants of $Y$ 
  \begin{align*}
  \langle \frac{1}{1-qL}\rangle_{0,1,1}^Y = \frac{32(q^2+q^3)}{(1-q)^4}.
  \end{align*}  

\end{document}